\documentclass[11pt]{article}

\usepackage{graphicx}
\usepackage{mathrsfs}
\usepackage{dsfont}
\usepackage{amsmath}
\usepackage{amsthm}
\usepackage{amsfonts}
\usepackage{amssymb}
\usepackage{verbatim}
\usepackage{alltt}
\input xy
\xyoption{all}
\usepackage{diagrams}
\usepackage[all]{xy}
\usepackage{ForestCMD}

\DeclareMathAlphabet{\mathpzc}{OT1}{pzc}{m}{it}

\theoremstyle{plain}
\newtheorem{theorem}{Theorem}[section]
\newtheorem{lemma}[theorem]{Lemma}
\newtheorem{proposition}[theorem]{Proposition}

\newtheorem{corollary}[theorem]{Corollary}
\newtheorem{definition}[theorem]{Definition}
\theoremstyle{definition}
\newtheorem{example}[theorem]{Example}
\newtheorem{remark}[theorem]{Remark}

\theoremstyle{definition}
\newtheorem*{acknowledgements}{Acknowledgements}
\renewenvironment{proof}{\noindent{\it Proof.}}{\bgroup\hspace{\stretch{1}}$\square$\egroup\medskip\par}

\newcommand{\RT}{\mathrm{\bf T}}

\newcommand{\RS}{\mathrm{\bf S}}
\newcommand{\CT}{{\bf \Omega}}
\newcommand{\R}{\mathbb{R}}
\newcommand{\Z}{\mathbb{Z}}

\newcommand{\rmap}{\longrightarrow}
\newcommand{\Rep}{\textrm{Rep}}

\newcommand{\uid}{\operatorname{id}}

\begin{document}

\title{Tensor products of representations up to homotopy}

\author{
Camilo Arias Abad\footnote{Institut f\"ur Mathematik, Universit\"at Z\"urich, Camilo.Arias.Abad@math.uzh.ch. Partially supported by NWO Grant ``Symmetries and Deformations in Geometry'',  SNF Grant 20-113439, the Forschungskredit of the  Universit\"at Z\"urich and the ESI in Vienna.}, Marius Crainic\footnote{Mathematics Institute, Utrecht University, M.Crainic@uu.nl. Partially supported by NWO grant 639.032.712.} and Benoit Dherin\footnote{Mathematics Institute, Utrecht University, B.R.U.Dherin@uu.nl. Partially supported by NWO Grant 613.000.602 and SNF Grant PA002-113136.}
}

\maketitle

\abstract{
We study the construction of tensor products of representations up to homotopy, which are the $A_\infty$ version of ordinary representations. 
We provide formulas for the construction of tensor products of representations up to homotopy and of morphisms between them, and show that these formulas
give the homotopy category a monoidal structure which is uniquely defined up to equivalence.
}

\tableofcontents

\section{Introduction}

This work is motivated by the study of the cohomology of classifying spaces of Lie groupoids. For a Lie group $G$,  Bott \cite{Bott} constructed a  
spectral sequence converging to the cohomology of the classifying space $BG$ with first page
\begin{equation}
 E_{1}^{pq}=H_{\mathrm{diff}}^{p-q}(G,S^{q}(\mathfrak{g^{*}})),\label{spectral sequence Bott}
\end{equation}
the differentiable cohomology with coefficients in the symmetric powers of the coadjoint representation. 
If the Lie group $G$ is compact then the first page of the spectral sequence vanishes outside of the diagonal, and one obtains that 
the cohomology of $BG$ is isomorphic to the invariant polynomials on the Lie algebra.
The Cartan model for equivariant cohomology can be seen as a generalization of this computation for classifying spaces of groupoids associated to
compact group actions on manifolds.  In \cite{Get} Getzler constructed a model for equivariant cohomology
of non compact groups, generalizing Bott's spectral sequence to the case of general group actions. 
Behrend \cite{Beh1} extended Getzler's model to the case of stacks that can be represented by ``flat groupoids''. For general Lie groupoids the situation is more subtle because the
``adjoint representation'' is no longer a representation in the usual sense. Instead, one has to work with the notion of
representation up to homotopy. We have shown in \cite{AC3}
that the Bott spectral sequence does exist for arbitrary Lie groupoids, provided one has a well-behaved operation of taking symmetric powers
of representations up to homotopy. In the present paper
we study the existence and the uniqueness of tensor products of representations up to homotopy.


For a small  category $\mathcal{C}$, the notion of representation up to homotopy  is the $A_\infty$ version of the usual notion of representation. 
In terms of $A_\infty$ structures, one associates to $\mathcal{C}$  the differential graded category $\mathbb{R}\mathcal{C}$ whose objects are
those of $\mathcal{C}$, and whose morphisms are the linear span over $\mathbb{R}$ of those of $\mathcal{C}$. Then, a representation up to homotopy of $\mathcal{C}$ is
an $A_\infty$ functor from $ \mathbb{R}\mathcal{C}$ to the $dg$-category of differential graded vector spaces. We 
will be interested in the case where $\mathcal{C}= G$ is a Lie groupoid and require the structure 
operators to be smooth in the appropriate sense.
We would like to point out that the assumption that $\mathcal{C}$ is a Lie groupoid does not play any role in the construction of tensor products. We chose this level of generality only because our original 
motivation comes from studying this case. However, the whole construction applies to arbitrary categories and, more generally, to twisting cochains over a simplicial set (see \cite{Stasheff}). These more general twisting cochains appear for instance in the parallel transport of superconnections discussed in \cite{Block-Smith} and \cite{Igusa}.


The works of Loday \cite{Loday}, Markl-Schnider \cite{MS}, Senablidze-Umble \cite{SU} and Stasheff \cite{Stasheff}  explain that the construction of
tensor products of higher homotopy algebraic structures amounts to the construction of certain ``diagonal maps'' in some appropriate family of polytopes. In the case of $A_\infty$-morphisms, the right family of polytopes is the multiplihedra (see \cite{Forcey}).  
Since we consider the case in which the domain and the range categories are strict, the polytopes controlling our problem become much simpler, indeed, they are cubes.
See also Sugawara \cite{Sug} and Forcey \cite{Forcey} for an account of this.
This simplification of the combinatorics allows us to study not only tensor products of representations up to homotopy, but also of morphisms between them.


Here is a short account of the results of this paper. We provide explicit and universal formulas for tensor products that are unital and strictly associative or strictly symmetric, while showing that any two tensor product operations are equivalent.
We explain how to take tensor products of morphisms between representations up to homotopy, which correspond to natural transformations between the $A_\infty$ functors. We prove that
once a choice has been made for taking tensor products of objects, there is a natural way to take tensor products of morphisms.
We show that these constructions produce monoidal structures on the homotopy category of the representations up to homotopy, and that
this monoidal structure is unique up to equivalence.

The category of representations up to homotopy of a Lie groupoid is naturally a $dg$-category and it seems natural to ask whether the monoidal structure on the
homotopy category can be lifted to this $dg$-category by making choices of tensor products of all lengths in a coherent way. This is an interesting question that we do not
address here.


We conclude this introduction with an outline of the paper. 

In \S2, we review the definitions of representations up to homotopy, the morphisms between them, and the homotopies between the morphisms. 

The purpose of \S3 is to isolate the algebraic structure that controls the problem of tensoring representations up to homotopy. 
We show that a representation up to homotopy of $G$ on a complex of vector bundles $E$ is the same thing as a Maurer-Cartan element in a certain DGA (differential graded algebra) 
$\bar{\mathcal{A}}_{E}$ associated to $G$ and $E$.  However, for the
purpose of handling tensor products, the structure of DGA is not fine enough: $\bar{\mathcal{A}}_{E\otimes F}$ cannot be expressed in terms of the DGAs $\bar{\mathcal{A}}_{E}$ and $\bar{\mathcal{A}}_{F}$.  For that reason, we introduce the finer notion of DB-algebra and we describe a functor $\bar{K}$ from the category of DB-algebras to the category of complete DGAs. The DGA $\bar{\mathcal{A}}_{E}$ comes from a canonical DB-algebra
$\mathcal{A}_{E}$ and, this time, $\mathcal{A}_{E\otimes F}$ is related to the tensor product of the DB-algebras $\mathcal{A}_{E}$ and $\mathcal{A}_{F}$. Thus, one can state the problem of
constructing tensor products of representations up to homotopy in the language of DB-algebras.

In \S4, we construct a DB-algebra $\Omega$ that is universal with respect to Maurer-Cartan elements in the sense that a morphism of DB-algebras $\Omega\rightarrow A$ is the same as a Maurer-Cartan element in the complete differential graded algebra $\bar{K}(A)$.

In \S5, we show that the problem of constructing tensor products of representations up to homotopy corresponds to finding certain Maurer-Cartan elements in some universal differential graded algebra. We prove the existence and uniqueness of these tensor products, and provide explicit formulas for strictly associative or strictly symmetric one. We show that a tensor product can not enjoy both of these properties at the same time. We also explain that the tensor product can be chosen so that the product of unital representations remains unital.

In  \S6, we explain how to take tensor products of morphisms between representations. We show that any two choices are homotopic. We prove that the homotopy category $\mathcal{D}(G)$ has a monoidal structure that is uniquely defined up to natural equivalence.

In \S7, we give a more concrete realization of tensor products of morphisms. We point out that any universal Maurer-Cartan element $\omega$ comes with a canonical endomorphism $x_{\omega}$. Hence,  once the choice of $\omega$ is made, no further choices are needed in order to take tensor products of morphisms. The construction of $x_{\omega}$ is based on a canonical Hochschild cocycle of degree zero on $\Omega$, which can be interpreted as the non-commutative DeRham differential of $\Omega$.

The appendix contains general facts about Maurer-Cartan elements in complete differential graded algebras, morphisms between them and their relationship to  Hochschild cohomology.

\begin{acknowledgements} 
We thank Andre Henriques, Jean Louis Loday, Ieke Moerdijk and Bruno Vallette for various conversations we had at different stages of this work. C.A.A. and B.D. also thank Calder Daenzer for suggesting we think about these questions in his cabin in lake Tahoe.
\end{acknowledgements}

\section{Representations up to homotopy}

In this section, we recall the definition of the category  of representations up to homotopy. As mentioned in the introduction, we work over general Lie groupoids $G$
(see e.g. \cite{Mackenzie} for the basics), but the reader may assume for simplicity that $G$ is a group. 

Hence  throughout this paper the letter $G$ will stand for a Lie groupoid (which we also identify with the space of arrows) over the base smooth manifold $M$
(the space of units). The source and target maps will be written as $s,t:G\rmap M$. 
A representation up to homotopy of $G$ consists of the following data: 

\begin{enumerate}
\item A graded vector bundle $E$ over $M$. 

\item A differential $\partial$ on $E$; that is, a degree-one vector bundle morphism 
\[
 \partial:E^{\bullet}\rmap E^{\bullet+1}
\]
with $\partial\circ\partial=0$. 

\item A smooth operator that associates to each $g\in G$ a chain map 
\[ 
 \lambda_{g}:E_{s(g)}\rmap E_{t(h)},\ e\mapsto\lambda_{g}(e),
\] 
which we will refer to as the \textit{quasi-action}. Here, \textit{quasi} refers to the fact that it may fail to respect the composition operation. 

\item A smooth operation that associates to each pair $(g,h)$ of composable arrows a homotopy between $\lambda_{g}\lambda_{h}$ and $\lambda_{gh}$; i.e., a linear map that lowers the cochain degree by one, 
\[
 R_{2}(g,h):E_{s(h)}\rmap E_{t(h)}
\] 
with the property that 
\begin{equation}
 \label{R2}
 \lambda_{g}\lambda_{h}-\lambda_{gh}=\partial(R_{2}(g,h)),
\end{equation}
where the last expression makes use of the induced differential in the $\textrm{Hom}$-bundle:
\[ 
 \partial(R_{2}(g,h))=[\partial,R_{2}(g,h)]=\partial\circ R_{2}(g,h)+R_{2}(g,h)\circ\partial.
\]
 
\item Similar higher-order operations $R_{k}$, in which each $R_{k}$ measures the failure of higher-coherence equations for $\partial,\lambda,R_{2} \ldots,R_{k-1}$. In order to have more uniform notation, we will often write $R_{0}=\partial$, $R_{1}=\lambda$. 
\end{enumerate}

For the the precise definition, we recall that a
string of $k$ composable arrows is a $k$-tuple $(g_{1},\ldots,g_{k})\in G^k$ of arrows satisfying $s(g_{i})=t(g_{i+1})$ for $i=1\dots,k$.

\begin{definition}\label{definitionunital}
A \textbf{representation up to homotopy} $(E,R_k)$ of a Lie groupoid $G$ is a graded vector bundle $E$ over the base $M$, together with a sequence of operations $R_{k}$, $k\geq0$, where $R_{k}$ associates to a string of $k$-composable arrows $(g_{1},\ldots,g_{k})$ a linear map
 \[
R_{k}(g_{1},\ldots,g_{k}):E_{s(g_{k})}\rmap E_{t(g_{1})},
\]
of degree $1-k$, depending smoothly on the arguments and satisfying the equations
\begin{equation}
 \label{structure equations}
 \sum_{j=1}^{k-1}(-1)^{j}R_{k-1}(g_{1},\ldots,g_{j}g_{j+1},\ldots,g_{k})=\sum_{j=0}^{k}(-1)^{j}R_{j}(g_{1},\ldots,g_{j})\circ R_{k-j}(g_{j+1},\ldots,g_{k}).
\end{equation}
The representation up to homotopy $(E,R_k)$  is said to be \textbf{unital} if the restriction of $R_1$ to the unit space $M$ is the vector bundle identity map $\uid_E$,  and if the higher components $R_k$ vanish when one of the arguments is a groupoid unit.
\end{definition}

We will denote the vector bundle morphism $R_{0}$ by $\partial^{E}$ or simply by $\partial$ when no confusion arises. We will say that $(E, \partial)$ is the complex underlying the representation up to homotopy $E$, or that the operators $\{R_{k}\}_{k\geq1}$ define a representation up to homotopy on the complex $(E,\partial)$. With this notation, the equations above read: 
\begin{eqnarray*}
 \sum_{j=1}^{k-1}(-1)^{j}R_{k-1}(g_{1},\ldots,g_{j}g_{j+1},\ldots,g_{k})+\sum_{j=1}^{k-1}(-1)^{j+1}R_{j}(g_{1},\ldots,g_{j})\circ R_{k-j}(g_{j+1},\ldots,g_{k})\\
 =\partial\circ R_{k}(g_{1},\ldots,g_{k})+(-1)^{k}R_{k}(g_{1},\ldots,g_{k})\circ\partial.
\end{eqnarray*}

We turn now to the definition of morphisms between representations up to homotopy:

\begin{definition}
\label{definition morphism} 
A \textbf{morphism} from a representation up to homotopy $(E,R_{k})$ to another one $(E',R_{k}^{'})$ is of a sequence $\Phi=\{\Phi_{k}\}_{k\geq0}$, where $\Phi_{k}$ is an operator that associates to a string of $k$-composable arrows $(g_{1},\ldots,g_{k})$ a linear map 
\[
 \Phi_{k}(g_{1},\ldots,g_{k}):E_{s(g_{k})}\rmap E_{t(g_{1})}'
\]
of degree $-k$, depending smoothly on the arguments, such that
\begin{eqnarray}\label{equations for a map}
 \sum_{i+j=k}(-1)^{j}\Phi_{j}(g_{1},\dots,g_{j})  \circ  R_{i}(g_{j+1},\dots,g_{k})&=&
  \sum_{i+j=k}R_{j}^{'}(g_{1},\dots,g_{j})\circ\Phi_{i}(g_{j+1},\dots,g_{k}) \\
 &  & +\sum_{j=1}^{k-1}(-1)^{j}\Phi_{k-1}(g_{1},\dots,g_{j}g_{j+1},\dots,g_{k}).\nonumber 
\end{eqnarray}
\end{definition}

The composition of morphisms is given by the formula
$$
(\Phi\circ\Psi)_k(g_1,\dots,g_k)=\sum_{i+j=k}\Phi_i(g_1,\dots,g_i)\circ\Psi(g_{i+1},\dots,g_k).
$$
We will denote by $\Rep^{\infty}(G)$ the resulting category of representations up to homotopy of $G$. Note that a morphism $\Phi$ is an isomorphism 
if and only if 
$\Phi_{0}$ is an isomorphism of (graded) vector bundles. We will also need the following stronger notion of isomorphism.

\begin{definition}
\label{def-strong} 
We say that two representations up to homotopy $(E,\partial,R_{k})$ and $(E',\partial',R_{k}^{'})$ are \textbf{strongly isomorphic} if $E=E'$, $\partial=\partial'$ and there exists a morphism $\Phi$ with $\Phi_{0}=\textrm{Id}_{E}$. 
In this case, $\Phi$ will be called a strong isomorphism. 
\end{definition}

There is also a natural notion of homotopy between morphisms: 

\begin{definition}
Let $\Phi$ and $\Psi$ be morphisms of representation up to homotopy from $(E,R_{k})$ to $(E',R_{k}^{'})$. A \textbf{homotopy} between $\Phi$ and $\Psi$ consists of a sequence $h=\{h_{k}\}$, where $h_{k}$ is an operator that associates with a string of $k$-composable arrows $(g_{1},\ldots,g_{k})$ a linear map 
\[
 h_{k}(g_{1},\ldots,g_{k}):E_{s(g_{k})}\rmap E_{t(g_{1})}',
\]
of degree $-k-1$ depending smoothly on the arguments, and such that
\begin{eqnarray}
\Phi_{k}-\Psi_{k} & = & \partial\circ h_{k}(g_{1},\dots,g_{k})+(-1)^{k}h_{k}(g_{1},\dots,g_{k})\circ\partial\label{equations for a homotopy}\\
 &  & +\sum_{i=0}^{k-1}(-1)^{i}h_{i}(g_{1},\dots,g_{i})\circ R_{k-i}(g_{i+1},\dots,g_{k})\nonumber\\
 &  & +\sum_{i=1}^{k}(-1)^{i}R'_{i}(g_{1},\dots,g_{i})\circ h_{k-i}(g_{i+1},\dots,g_{k})\nonumber \\
 &  & +\sum_{i=1}^{k-1}(-1)^{i+1}h_{k-1}(g_{1},\dots,g_{i}g_{i+1},\dots,g_{k}).\nonumber 
\end{eqnarray}
\end{definition}

The composition is well defined on homotopy classes of morphisms. The \textbf{homotopy category} $\mathcal{D}(G)$ is defined as the category whose objects are representations up to homotopy and whose morphisms are homotopy classes of morphisms between representations up to homotopy. 

\vspace*{.2in}

Let us now describe the problem of constructing tensor products. Given two representations up to homotopy $E$ and $F$, the tensor product $E\otimes F$ is defined, first of all, as a cochain complex of vector bundles over $M$ with the standard differential 
\[
 \partial(e\otimes f)=\partial(e)\otimes f+(-1)^{p}e\otimes\partial(f),
\]
where $p$ is the degree of $e$. The first step toward giving this complex the structure of a representation up to homotopy is to define the $R_{1}$-term. Thinking of it as a quasi-action, there is again a standard choice, the diagonal one: 
\[
 \lambda_{g}(e\otimes f)=\lambda_{g}(e)\otimes\lambda_{g}(f).
\]
However, for higher $R$'s, the problem is more subtle. For instance, when looking for an $R_{2}$, we have to make sure that the equation (\ref{R2}) for $E\otimes F$ is satisfied. Already in this case the equation has more than one natural and interesting solution. For instance, if one is interested in a symmetric tensor product, then there is only one solution for $R_{2}$: 
\begin{eqnarray*}
 R_{2}(g,h)(e\otimes f) & = & \frac{1}{2}(  R_{2}(g,h)(e)\otimes\lambda(gh)(f)+R_{2}(g,h)(e)\otimes(\lambda(g)\circ\lambda(h))(f)\\
 &  & +\lambda(gh)(e)\otimes R_{2}(g,h)(f)+(\lambda(g)\circ\lambda(h))(e)\otimes R_{2}(g,h)(f)).
\end{eqnarray*}
On the other hand, this specific second component would not work if we wanted the tensor product to be associative, in which case we could choose, for instance, the second component to be 
\begin{eqnarray*}
 R_{2}(g,h)(e\otimes f) & = & R_{2}(g,h)(e)\otimes\lambda(gh)(f)+(\lambda(g)\circ\lambda(h))(e)\otimes R_{2}(g,h)(f)).
\end{eqnarray*}
For higher values of $k$, the equations become much more involved. The aim of this paper is to understand the algebraic structure that
governs representations up to homotopy, and to use it to classify all possible tensor products of representations up to homotopy and of morphisms between them.

\section{Maurer-Cartan elements and DB-algebras}

\label{Maurer-Cartan elements and DB-algebras}

In this section, we discuss the algebraic structures that are relevant to the construction of tensor products. First, we interpret representations up to homotopy as Maurer-Cartan elements in a certain DGA (Differential Graded Algebra) and then we describe the building pieces of the DGAs involved. This underlying algebraic structure is important when tensoring two representations up to homotopy, and we axiomatize it under the name of DB-algebra (Differential Bar algebra). Hence the main outcome is the construction of functors
\[
 \Big\{\textrm{Complexes of vector bundles}\ (E,\partial)\ \textrm{over}\ M\Big\}
  \rmap 
 \Big\{\textrm{DB-algebras}\Big\}
  \rmap 
 \Big\{\textrm{Complete\ DGAs}\Big\},
\]
\[ \ \ \ E\mapsto\mathcal{A}_{E}  \mapsto \bar{\mathcal{A}}_{E},\]
so that representations up to homotopy on $E$ correspond to Maurer-Cartan elements in $ \bar{\mathcal{A}}_{E}$
and the resulting composition functor behaves well with respect to tensor products. 
Moreover, the notion of strong isomorphism on the left hand side corresponds to the notion of gauge equivalence between Maurer-Cartan 
elements on the right hand side. For the general notion of complete DGAs, Maurer-Cartan elements and
gauge equivalences, we refer the reader to the appendix.

\subsection{Representations up to homotopy as Maurer-Cartan elements}

We start by constructing $\bar{\mathcal{A}}_{E}$. Let $G_k$ be the submanifold of $G^k$ consisting of 
strings of $k$ composable arrows of $G$. By convention, $G_0 = M$. We will also denote by $s$ and $t$ the maps 
$G_k \rightarrow M$ given by $s(g_1,\dots, g_k)=s(g_k)$ and $t(g_1,\dots,g_k)=t(g_1).$
For a graded vector bundle $E$ over $M$, we consider the pull-back bundles $s^{*}E$ and $t^{*}E$ to $G_{k}$, and we form the graded $\textrm{Hom}$-bundle $\textrm{Hom}(s^{*}E,t^{*}E)$ over
$G_{k}$. Recall that $\phi:s^{*}E\rmap t^{*}E$ has degree $l$ if it maps $s^{*}(E^{\bullet})$ to $t^{*}(E^{\bullet+l})$. We will consider the resulting spaces of sections
\begin{equation} 
 \label{space of sections}
 \mathcal{A}_{E}^{k}(l):=\Gamma(G_{k},\textrm{Hom}^{l}(s^{*}E,t^{*}E)).
\end{equation}
For $c\in\mathcal{A}_{E}^{k}(l)$, we write 
\[
 k(c)=k,\ l(c)=l,\ |c|=k(c)+l(c),
\]
and we call $|c|$ the total degree of $c$. All these spaces together define a bigraded algebra, with the product $\star$ that associates to $c\in\mathcal{A}_{E}^{k}(l)$ and $c'\in\mathcal{A}_{E}^{k'}(l')$ the element $c\star c'\in\mathcal{A}_{E}^{k+k'}(l+l')$, given by 
\begin{equation}
 (c\star c')(g_{1},\ldots,g_{k+k'})=(-1)^{k(k'+l')}c(g_{1},\ldots,g_{k})\circ c'(g_{k+1},\ldots,g_{k+k'}).\label{formula-star}
\end{equation}

When $E$ is a cochain complex, then so is the $\textrm{Hom}$-bundle, with the differential
\begin{equation} 
 \partial(\phi)=[\partial,\phi]=\partial\circ\phi-(-1)^{l}\phi\circ\partial,
\end{equation}
for $\phi\in\textrm{Hom}^{l}(s^{*}E,t^{*}E)$. This defines a differential
\begin{equation} 
 \label{partial in A_E} 
 \partial:\mathcal{A}_{E}^{k}(l)\rmap\mathcal{A}_{E}^{k}(l+1),
\end{equation}
induced by the differential of $E$. On the other hand, the groupoid structure induces a differential along the other degree: 
\[
 d:\mathcal{A}_{E}^{k}(l)\rmap\mathcal{A}_{E}^{k+1}(l),
\]
\[
 d(c)(g_{1},\ldots,g_{k+1})=\sum_{j=1}^{k}(-1)^{j}c(g_{1},\ldots,g_{j}g_{j+1},\ldots,g_{k+1}).
\]
We denote by $\bar{\mathcal{A}}_{E}$ the DGA that, in degree $n$, is given by 
\[
 \bar{\mathcal{A}}_{E}^{n}:=\Pi_{k+l=n}\mathcal{A}_{E}^{k}(l)
\]
and whose elements should be thought of as infinite sums $\gamma_0+\gamma_1+\gamma_2+\cdots$ of homogeneous elements $\gamma_i\in\bar{\mathcal{A}}_{E}^{i}$, $i\geq 0$. The product $\star$ and the total differential 
\[
 d_{\textrm{tot}}(c):=\partial(c)+(-1)^{n}d(c)
\]
give $\bar{\mathcal{A}}_{E}$ the structure of a DGA. The signs are chosen so that the differential is a derivation with respect to
$\star$. This DGA is a {\it complete} DGA, in the sense of the appendix, with the filtration: 
\[
 F_{p}\bar{\mathcal{A}}_{E}\;:=\;\Big\{\gamma=\gamma_{0}+\gamma_{1}+\ldots\in\bar{\mathcal{A}}_{E}:
    \gamma_{0}=\ldots= \gamma_{p-1}=0\Big\}.
\]
Note that the structure of $\bar{\mathcal{A}}_{E}$ depends on the differential $\partial$ and not only on the vector bundle $E$. 
The formulas that appear in the definition of representations up to homotopy and of morphisms between them take now the following
more compact form, which follows by a direct computation.

\begin{proposition} \label{rep-MC} 
Let $G$ be a Lie groupoid over a manifold $M$ and let $(E,\partial)$ a cochain complex of vector bundles over $M$. Also, let $\{R_{k}\}_{k\geq1}$ be a sequence of operators such that $R_{k}\in\mathcal{A}_{E}^{k}(1-k)$. Then, $(E,\partial,R_{k})$ is a representation up to homotopy of $G$ if and only if
\[ 
 R_{E}:=R_{1}+R_{2}+\ldots\ \in\bar{\mathcal{A}}_{E}
\]
is a Maurer-Cartan element for $\bar{\mathcal{A}}_{E}$. Moreover, for two such sets of operations $\{R_{k}\}_{k\geq1}$ and $\{R_{k}^{'}\}_{k\geq1}$, there is a one-to-one correspondence between: 
\begin{enumerate}
 \item strong isomorphisms between $(E,\partial,R_{k})$ and $(E,\partial,R_{k}^{'})$ (Definition \ref{def-strong}), and
 \item strong gauge equivalences between the Maurer-Cartan elements $R_{E},R'_{E}\in\bar{\mathcal{A}}_{E}$ (Definition \ref{strong-appendix}). 
\end{enumerate}
\end{proposition}

\subsection{DB-algebras}

The description of representations up to homotopy in terms of Maurer-Cartan elements is still not very useful when it comes to constructing tensor products. The reason is very simple: given two cochain complexes $E$ and $F$, the DGA $\bar{\mathcal{A}}_{E\otimes F}$ is not directly related to the tensor product of the DGAs $\bar{\mathcal{A}}_{E}$ and $\bar{\mathcal{A}}_{F}$. Looking at the differential of $\bar{\mathcal{A}}_{E\otimes F}$ it becomes clear that
there is more structure present in $\bar{\mathcal{A}}_{E}$ and $\bar{\mathcal{A}}_{F}$ then just that of DGA. This brings us to the notion of a DB-algebra.

\begin{definition}\label{T-algebras} 
A differential bar-algebra, or \textbf{DB-algebra}, is a bigraded vector space
\[
 \mathcal{A}=\bigoplus_{k\geq0,l\in\mathbb{Z}}\mathcal{A}^{k}(l)
\]
together with: 
\begin{itemize}
 \item A structure of bigraded associative algebra with the product 
 \[ 
  \circ:\mathcal{A}^{k}(l)\otimes\mathcal{A}^{k'}(l')\rmap\mathcal{A}^{k+k'}(l+l').
 \] 
 For $a\in\mathcal{A}^{k}(l)$, we write $k(a)=k$, $l(a)=l$, and we define the total degree $|a|=k(a)+l(a)$. 
 
 \item A derivation of bidegree $(1,0)$; i.e., a linear map  
 \[ 
  \partial:\mathcal{A}^{k}(l)\rmap\mathcal{A}^{k}(l+1)
 \] 
 that satisfies 
 \begin{equation}
  \partial(a\circ b)=\partial(a)\circ b+(-1)^{l(a)}a\circ\partial{b}.\label{T-algebras-Leibniz}
 \end{equation}
  
 \item For each $k\geq1$, there are linear maps 
 \[
  d_{i}:\mathcal{A}^{k}(l)\rmap\mathcal{A}^{k+1}(l),\quad i=1,\dots,k,
 \]
 commuting with $\partial$ and satisfying 
 \begin{eqnarray*}
  d_{j}d_{i} & = & d_{i}d_{j-1},\quad\textrm{if}\; i<j,
 \end{eqnarray*}
 and, for $a\in\mathcal{A}^{k}(l)$, 
 \begin{equation}
  d_{i}(a\circ b)=\begin{cases}
  d_{i}(a)\circ b,\quad k\geq i\\
  a\circ d_{i-k}(b),\quad k<i.\end{cases}\label{T-algebras-di}
 \end{equation}
\end{itemize}
A morphism between two DB-algebras is a linear map that preserves both degrees and commutes with all the structure maps. We denote by $\underline{\mathcal{DB}ar}$ the resulting category. 
\end{definition}

For a general DB-algebra $\mathcal{A}$ we introduce the operators
\begin{equation}
 d=\sum_{i=1}^{k}(-1)^{i}d_{i}:\mathcal{A}^{k}(l)\rmap\mathcal{A}^{k+1}(l).\label{d}
\end{equation}
From the axioms, it follows that $d$ is a biderivation with respect to $\circ$.

\begin{lemma}\label{passing-to-DGA} 
Let $\mathcal{A}$ be a DB-algebra. Then $\mathcal{A}$, together with the total grading, the signed product
\begin{equation}
 a\star b=(-1)^{k(a)|b|}a\circ b \label{bullet}
\end{equation}
and the total differential 
\[ 
 d_{\textrm{tot}}=\partial+(-1)^{n}d:\mathcal{A}^{n}\rmap\mathcal{A}^{n+1},
\] 
is a DGA. 
\end{lemma}

\begin{definition} 
Given a DB-algebra $\mathcal{A}$, we denote by $\bar{\mathcal{A}}$ the completion of $\mathcal{A}$ with respect to the filtration by the $k$-degree. In other words, $\bar{\mathcal{A}}$ is the DGA with 
\[
 \bar{\mathcal{A}}^{n}=\prod_{k+l=n}\mathcal{A}^{k}(l),
\]
endowed with $\star$ and $d_{\textrm{tot}}$. The elements $a\in\bar{\mathcal{A}}^{n}$ will be written as infinite sums 
\begin{equation}
 a=a_{0}+a_{1}+\ldots,\ \ \ \textrm{with}\ a_{k}\in\mathcal{A}^{k}(n-k),\label{series}
\end{equation}
and we call $a_{k}$ the $k$-th component of $a$. This construction defines a functor 
\[
 \overline{K}:\underline{\mathcal{DB}ar}\rmap\overline{DGA}
\]
from the category of DB-algebras to the category of complete DGAs.
\end{definition}

\begin{example} \label{explanation} 
It is now clear that the DGA $(\bar{\mathcal{A}}_E,d_{tot},\star)$ from the previous subsection comes from a 
DB-algebra $(\mathcal{A}_E,\partial,\circ, d_i)$:
\begin{itemize} 
 \item the underlying bigraded space is $ \bigoplus_{k\geq 0,\,l \in \Z} \mathcal{A}_E^k(l)$,
 \item the product $\circ$ is the unsigned version of $\star$:
   \begin{equation}
    (c\circ c')(g_{1},\ldots,g_{k+k'})=c(g_{1},\ldots,g_{k})\circ c'(g_{k+1},\ldots,g_{k+k'}),\label{composition-formula}
   \end{equation}
 \item the differential $\partial$ is the $\textrm{Hom}$-bundle differential defined in \eqref{partial in A_E}, and
 \item the operators $d_{i}$ are given by the formulas:
\[
 d_{i}(c)(g_{1},\ldots,g_{k+1})=c(g_{1},\ldots,g_{i}g_{i+1},\ldots,g_{k+1}).
\]
\end{itemize}
\end{example}

\subsection{The tensor product of DB-algebras}
\label{The tensor product of DB-algebras}

The category $\underline{\mathcal{DB}ar}$ has a natural tensor product operation that will be denoted by $\boxtimes$. Given two DB-algebras $\mathcal{A}$ and $\mathcal{B}$, their tensor product $\mathcal{A}\boxtimes\mathcal{B}$ is defined as follows. As a bigraded vector space, 
\[
 (\mathcal{A}\boxtimes\mathcal{B})^{k}(l)=\bigoplus_{i+j=l}\mathcal{A}^{k}(i)\otimes\mathcal{B}^{k}(j).
\]
For $a\in\mathcal{A}^{k}(i)$ and $b\in\mathcal{B}^{k}(j)$, we will denote by $a\boxtimes b$ the resulting tensor in $\mathcal{A}\boxtimes\mathcal{B}$. The differential $\partial$ and the operators $d_{i}$ are given by 
\[
 \partial(a\boxtimes b)=\partial(a)\boxtimes b+(-1)^{l(a)}a\boxtimes\partial(b),\ d_{i}(a\boxtimes b)=d_{i}(a)\boxtimes d_{i}(b),
\]
while the multiplication $\circ$ by 
\[
 (a\boxtimes b)\circ(a'\boxtimes b')=(-1)^{l(b)l(a')}(a\circ a')\boxtimes(b\circ b').
\]

The previous definition is designed so that the construction $E\mapsto\mathcal{A}_{E}$ behaves well with respect to tensor products.

\begin{proposition}\label{mEF} 
For any two complexes of vector bundles $E$ and $F$ over $M$, the canonical map
\[
 m_{E,F}:\mathcal{A}_{E}\boxtimes\mathcal{A}_{F}\rmap\mathcal{A}_{E\otimes F},
\]
\[
 m_{E,F}(c\boxtimes c')(g_{1},\ldots,g_{k})=c(g_{1},\ldots,g_{k})\otimes c'(g_{1},\ldots,g_{k})
\]
is a morphism of DB-algebras.
\end{proposition}

For later use we mention here that, for any DB-algebra $\mathcal{A}$, there is a natural action of the group $S_{m}$ on $\mathcal{A}^{\boxtimes m}$. For $\sigma\in S_{m}$, the associated automorphism of $\mathcal{A}^{\boxtimes m}$ is denoted by $\hat{\sigma}$. To define $\hat{\sigma}$, it suffices to describe it when $\sigma=\tau_{i,i+1}$ is a transposition that interchanges the positions $i$ and $i+1$; in this case: 
\[
 \hat{\sigma}(a_{1}\boxtimes\ldots\boxtimes a_{m})=(-1)^{ll'}a_{1}\boxtimes\ldots\boxtimes a_{i-1}\boxtimes a_{i+1}\boxtimes a_{i}\boxtimes \ldots\boxtimes a_{m},
\]
for $a_{i}\in\mathcal{A}^{k}(l)$, $a_{i+1}\in\mathcal{A}^{k'}(l')$. It is not difficult to see that this defines 
an action of $S_{m}$ on $\mathcal{A}^{\boxtimes m}$ by automorphisms of DB-algebras.

\section{The Maurer-Cartan DB-algebra}


As explained in Proposition \ref{rep-MC}, representations up to homotopy structures on a complex of vector bundles correspond to Maurer-Cartan elements on the associated DGA. This observation allows one to translate the problem of constructing tensor products of representations up to homotopy to that of constructing Maurer-Cartan elements on the DGAs associated to tensor products of DB-algebras. Clearly, this problem can be treated at a universal level. This brings us to the Maurer-Cartan DB-algebra, which is the universal DB-algebra for Maurer-Cartan elements.

\begin{definition} \label{def-abstr-Omega}
For a DB-algebra $\mathcal{A}$, we denote by $MC_{1}(\bar{\mathcal{A}})$ the set of Maurer-Cartan elements of $\bar{\mathcal{A}}$ whose zeroth component vanishes. A \textbf{Maurer-Cartan algebra} is a DB-algebra $\Omega$, together with a Maurer-Cartan element $L\in MC_{1}(\bar{\Omega})$ with the property that for any DB-algebra $\mathcal{A}$, the map 
\[
 \textrm{Hom}_{\underline{\mathcal{DB}ar}^{}}(\Omega,\mathcal{A})\rmap MC_{1}(\bar{\mathcal{A}}),\phi\mapsto\phi(L)
\]
is a bijection. 
\end{definition}

\begin{theorem}\label{MC-exists} 
The Maurer-Cartan DB-algebra exists and is unique up to isomorphisms of DB-algebras. Moreover, for each $k$, $H^{l}(\Omega^{k}(\bullet),\partial)=0$ for all $l\neq0$.
\end{theorem}

The uniqueness follows by standard arguments. The aim of this section is to provide several explicit descriptions of $\Omega$, proving in particular the theorem above. 
The main conclusion of this section is the resulting reformulation of the notion of representation up to homotopy in terms of $\Omega$: 

\begin{corollary} \label{cor-rep-omega} Given a Lie groupoid $G$ over $M$
and a complex of vector bundles $(E,\partial)$, there is a one-to-one correspondence between sequences of operations $R=\{R_{k}\}_{k\geq1}$ making $(E,\partial,R_{k})$ into a representation up to homotopy of $G$ and morphisms of DB-algebras
\[ k_{E,R}:\Omega\rmap\mathcal{A}_{E}.\]
\end{corollary}

The map $k_{E,R}$, also denoted $k_{E}$,  will be called the \textbf{characteristic map} of the representation up to homotopy $(E,R)$.

\subsection{Description in terms of trees}

In our construction of $\Omega$, instead of proceeding abstractly and use generators and relations, we follow 
a pictorial approach. We start by explaining the main idea. Due to the expected universal property of $\Omega$, a representation
up to homotopy $(E, R_{E})$ is represented by its characteristic map 
$k_{E}: \Omega\rmap \mathcal{A}_{E}$ --uniquely determined by the fact that it sends the component $L_n$ of $L$ to the operation $R_{E}^{n}$. Hence 
general elements $A$ of $\Omega$ should encode certain operations 
$R_{E}^{A}$ on $E$ which arise by combining all the given operations $R_{E}^{n}$. 
There are various such operations one can think of. For instance, one has the following:
\begin{equation}
\label{oper-ex} 
(g_1, g_2, g_3, g_4, g_5, g_6)\mapsto R_{E}^2(g_1, g_2g_3) \circ R_{E}^{1}(g_4g_5)\circ R_{E}^{1}(g_6).
\end{equation}
The idea is to encode such operations graphically, by forests of height two. For instance, the operation above
is encoded by:
\begin{eqnarray*}
\BigFor
\end{eqnarray*}
and one should think of the six leaves as labelled by the six elements $g_1, \ldots , g_6$.

\begin{definition} 
We denote by $\RT$ the set of isomorphism classes of planar rooted trees whose leaves all have height $2$. We denote by $\RS$ the set of short forests; that is, the set of finite tuples $(T_{1,}\dots,T_{n})$  of trees in $\RT$. 
\end{definition}

We represent a short forest by joining the roots of the $T_{i}$'s by a horizontal line. For instance, 
\[
 \BigFor\quad\textrm{stands for}\quad\left(\quad\ForestBA\quad,\quad\ForestA\quad,\quad\ForestO\quad\right).
\]

Next, we introduce a bigrading on $\RS$.

\begin{definition} 
A branch of a short forest $F\in\RS$ is an edge that goes from a root to a vertex that is not a root. For any short forest $F$, we define 
\begin{eqnarray*}
 k(F) & = & \#\text{ of leaves of }F,\ \textrm{called\ the\ order\ of}\ F,\\
 b(F) & = & \#\text{ of branches of }F,\\
 r(F) & = & \#\text{ of roots of }F,\\
 l(F) & = & r(F)-b(F),\ \textrm{called\ the\ degree\ of}\ F.
\end{eqnarray*}
We denote by $\RS^{k}(l)$ the set of short forests of order $k$ and degree $l$.
\end{definition}

\begin{example} 
For the tree mentioned above,  
\[
 F=\;\BigFor\;,\; k(F)=6,\, b(F)=4,\, r(F)=3,\, l(F)=-1.
\]
The fact that $F\in S^6(-1)$ corresponds to the fact that the operation (\ref{oper-ex}) belongs to $\mathcal{A}_{E}^{6}(-1)$.
\end{example}

\begin{definition}
We denote by $(\CT,\circ)$ the free algebra over $\mathbb{R}$ generated by the trees in $\RT$ or, equivalently, the linear span over $\mathbb{R}$ of $\RS$ with the product given by the concatenation. 
\end{definition}

Pictorially, $F\circ F'$ is the forest obtained by joining the roots by an edge, as in the following example:
\[
 \ForestBA\,\,\circ\,\,\ForestAC\,\,=\,\,\BigFor\,\,\,.
\]
The bigrading on $\RS$ induces a similar bigrading on $\CT$ and allows us to talk about the spaces $\CT^{k}(l)$. 

\begin{definition} \label{def-delta} 
For each $i=1,\dots,k$, we define the operator
\[
 d_{i}:\CT^{k}(\bullet)\rightarrow\CT^{k+1}(\bullet),
\]
which acts by replacing the $i$-th leaf of a forest, counting from the left, by two leaves.

For each $l$, we define $\partial:\CT^{\bullet}(l)\rightarrow\CT^{\bullet}(l+1)$ by 
\begin{equation}
 \partial(F)=\sum_{j=1}^{l}(-1)^{j+1}\big(\partial_{j}^{1}F-\partial_{j}^{0}F\big), \label{differential formula}
\end{equation}
where $\partial_{j}^{1}F$ is obtained by separating the $j^{th}$ pair of adjacent branches (counted from left to right) and $\partial_{j}^{0}F$ by collapsing the $j^{th}$ pair of adjacent branches (counted from left to right).

Finally, we denote by $L_{n}$ the tree in $\CT$ that has one root, $n$ branches, and $n$ leaves:
\[
 L_{1}\;=\;\uno\;,\quad L_{2}\;=\;\dos\;,\quad L_{3}\;=\;\tres\;,\quad L_{4}\;=\;\quatro\;\quad\ldots,
\]  
and we set
\[
 L\; := \;L_{1}+L_{2}+\ldots\in\bar{\CT}.
\]
\end{definition}

\begin{example} 
Here is an example of the action, on the short forests, of the $d_i$ operators,
\[
 d_{2}\left(\,\,\,\ForestBAC\,\,\,\right)\,\,\,=\,\,\,\BigForB\,\,\,
\]
and, here, of the differential,
\[
 \partial\left(\,\,\ForestBAB\,\,\right)=\left(\,\,\,\,\ForestCAB\,\,\,\,\,-\,\,\ForestAAB\,\,\,\,\,\right)-\left(\,\,\,\,\ForestBAC\,\,\,-\,\,\,\,\ForestBAA\,\,\,\,\right).
\]
\end{example}

\subsection{Description in terms of words in three letters}

The set $\RS^{k+1}$ of short forests with $k+1$ leaves can be naturally identified with the set of words in three letters $\{a,b,c\}$ of length $k$ as follows. To a short forest $F$ with $k+1$ leaves $l_{1},\dots,l_{k+1}$, numbered from left to right, we associate the word $e_{1}\dots e_{k}$, where: 
\[
e_{i}=
\begin{cases}
 a & \text{if $l_{i}$ and $l_{i+1}$ belong to the same branch},\\
 b & \text{if $l_{i}$ and $l_{i+1}$ belong to different branches of the same tree},\\
 c & \text{if $l_{i}$ and $l_{i+1}$ belong to different trees.}
\end{cases}
\]
The unique short forest in $S(1)$, 
\[
 e:=\,\,\,\uno\;\;,
\]
corresponds to the empty word. The following figure illustrates this correspondence:

\medskip{}

\begin{center}
\includegraphics[scale=0.4]{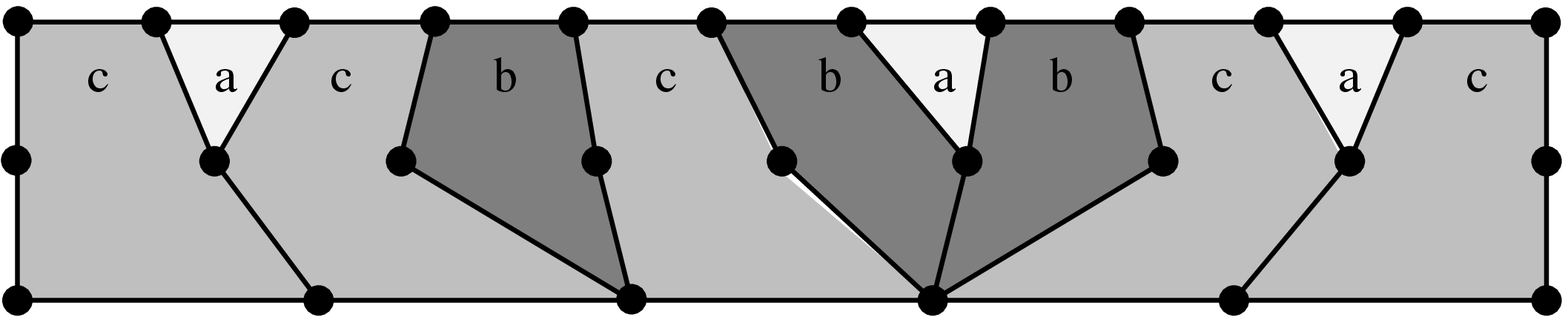} 
\par
\end{center}

\medskip{}

This construction identifies $\Omega$, as a vector space, with the free unital algebra  $F_{e}\langle a,b,c\rangle$ on the generators $a$, $b$ and $c$ with unit $e$. In terms of short forests: 
\[
 a=\,\,\,\ForestA\,\,\,,\quad b=\,\,\,\ForestB\,\,\,,\quad c=\,\,\,\ForestC\,\,\:,\quad e=\,\,\,\ForestO\,\,\,.
\]
The natural product coming from the concatenation of words will be denoted by $\diamond$. Note that $\diamond$ does not coincide with the product $\circ$ defined on $\Omega$. In terms of forests, $F\diamond F'$ is obtained by identifying the rightmost branch and leaf of $F$ with the leftmost branch and leaf of $F'$, as in the following example:
\[
 \ForestBA\,\,\diamond\,\,\ForestC\,\,=\,\,\ForestBAC\,\,.
\]

In terms of words on three letters, the algebraic structure of $\Omega$ has the following description: 

\begin{itemize}
 \item The product on $\Omega$ is given by: 
 \[
  T\circ T'=T\diamond c\diamond T'.
 \]
 
 \item The operator $d_{i}$ acts according to the formula: 
 \[
  d_{i}(e_{1}\ldots e_{k})=e_{1}\dots e_{i-1}ae_{i}\ldots e_{k}.
 \]
 
 \item The operator $\partial$ is the unique derivation with respect to $\diamond$; i.e., with the property that 
 \begin{eqnarray}
  \partial(F\diamond F') & = & \partial(F)\diamond F'+(-1)^{l(F)}F\diamond\partial(F'),\label{derivation}
 \end{eqnarray}
 given on generators by
 \begin{eqnarray}
  \partial\left(e\right)=\partial\left(a\right)=\partial\left(c\right)=0 &\textrm{and}& \partial(b)=c-a. \label{generators}
 \end{eqnarray}

 \item $L_{n}=b^{\diamond n-1}$ for $n\geq 2$,  and $L_1 = e$. 
\end{itemize}
The only statement in the list above that requires a proof is the derivation property \eqref{derivation}. Denote by $\hat{\partial}$ the operator defined on the generators as in \eqref{generators} and extended by derivation as in \eqref{derivation}. We need to prove that $\hat{\partial}=\partial$. Consider a word $T=e_{1}\diamond\ldots\diamond e_{k}$ with $e_{i}\in\{a,b,c\}$; then: 
\begin{equation}
 \hat{\partial}(T)=\sum_{e_{i}=b}(-1)^{l(w_{i})}e_{1}\diamond\dots\diamond  e_{i-1}\diamond(c-a)\diamond e_{i+1}\diamond\dots\diamond e_{k},
\end{equation}
where $l(w_{i})$ is the degree of $w_{i}=e_{1}\diamond\dots\diamond e_{i-1}$. On the other hand, 
\[
 e_{1}\diamond\dots\diamond e_{i-1}\diamond c\diamond e_{i+1}\diamond\dots\diamond e_{k}=\partial_{j}^{0}(T),
\]
\[
 e_{1}\diamond\dots\diamond e_{i-1}\diamond a\diamond e_{i+1}\diamond\dots\diamond e_{k}=\partial_{j}^{1}(T).
\]
Thus, from the formula in Definition \ref{def-delta}, we conclude that $\hat{\partial}=\partial$.

\subsection{Description in terms of faces of cubes \label{sub: Cube Description }}

The DB-algebra $\Omega$ can also be constructed in terms of the faces of cubes. Namely, the set of words in three letters can be naturally identified with the set of faces of the geometric cubes $[0,1]^{k}\subset\mathbb{R}^{k}$
as follows. To a word $F=e_{1}\dots e_{k}$, we associate the face
\[
\phi(F)=\psi(e_{1})\times\cdots\times\psi(e_{k})\subset[0,1]^{k},\]
with the convention that 
\begin{eqnarray*}
 \psi(e) & = & \left\{ \begin{array}{cl}
       \{0\} & \textnormal{if }e=a,\\
       \{1\} & \textnormal{if }e=c,\\
       {}[0,1] & \textnormal{if }e=b.
      \end{array}\right.
\end{eqnarray*}
This simply says that the cells of the cube $I^{k}$ are products of the cells of the interval, which we label as follows: $\{0\}=a,(0,1)=b$
and $\{1\}=c$. Thus, we identify the cells of $I^{k}$ with words of length $k$ in the letters $\{a,b,c\}$. Note that the dimension of a cell is the number of times that $b$ appears in the corresponding word. The following figure illustrates this bijection:

\medskip{}

\[
\ParamSquare{\ForestBC}{\ForestAC}{\ForestCC}{\ForestAB}{\ForestBB}{\ForestCB}{\ForestAA}{\ForestCA}{\ForestBA}
\]

\medskip{}

In this correspondence, a forest of degree $-l$ with $k+1$ leaves is sent to a $l$-dimensional face of the $k$-dimensional cube. Thus,
we obtain the identification \begin{eqnarray}\big(\Omega^{k+1}(\bullet),\partial\big) \cong \big(C_{\bullet}(I^{k}),\partial\big),\, k\geq0,\label{Cubes-Forests}\end{eqnarray} where $C_{\bullet}(I^{k})$ is the cellular chain complex computing the homology of the $k$-dimensional cube with respect to the natural cell decomposition and negative grading of the cells. Also, one easily shows that: 

\begin{itemize}
 \item The product $F_{1}\diamond F_{2}$ corresponds to the Cartesian product of cells $\phi(F_{1})\times\phi(F_{2})$. This also shows that the product $F_{1}\circ F_{2}$ corresponds to the operation $\phi(F_{1})\times\{1\}\times\phi(F_{2})$. 
 \item The operator $\partial$ corresponds to the boundary operator in $C_\bullet(I^{k})$. 
 \item $L_{k}$ corresponds to the highest degree cell in $I^{k-1}$. 
 \item The operators $d_{i}$ correspond to the various ways of embedding the $k$-cube into the $(k+1)$-cube as a $k$-face having the origin as one of its vertices. This is illustrated in the figure below:
\end{itemize}

\begin{center}
\includegraphics[scale=0.3]{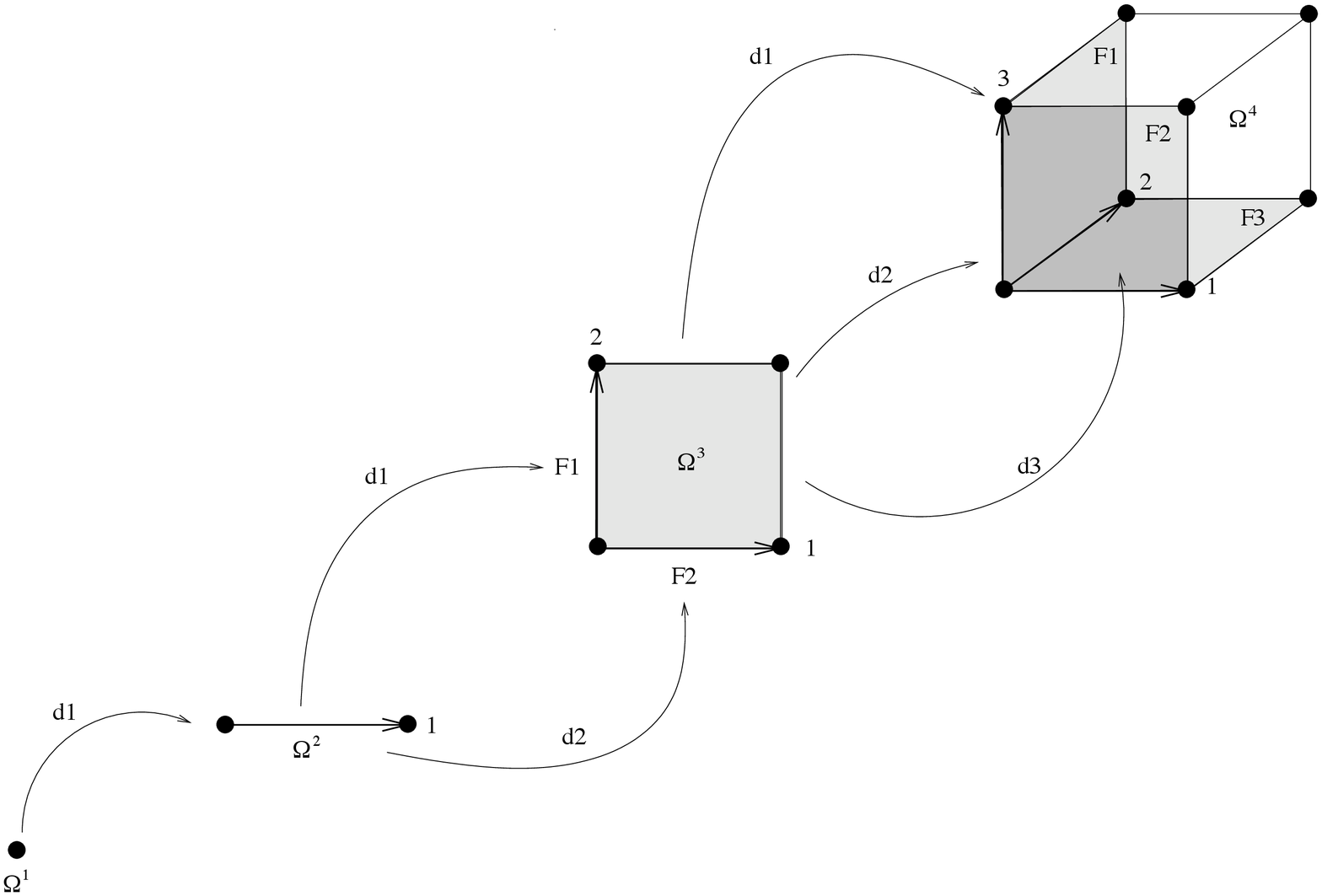} 
\par\end{center}

\subsection{Proof of Theorem \ref{MC-exists}}

We will now prove that $(\CT,\circ,d_{i},\partial)$, together with
\[
 L=L_{1}+L_{2}+\dots,
\]
satisfies the universal property of the Maurer-Cartan DB-algebra. The fact that $\CT$ is a DB-algebra is quite straightforward now. For instance, to check that $\partial^{2}=0$, one uses the description of $\partial$ as a derivation with respect to $\diamond$, and one is left with checking this equation on the elements $a$, $b$ and $c$. We should also prove that $\partial$ is a derivation with respect to the product $\circ$. Using the expression of $\circ$ in terms of $\diamond$, we find: 
\begin{eqnarray*}
 \partial(T\circ T') & = & \partial(T\diamond c\diamond T')=\partial(T)\diamond c\diamond T'+(-1)^{l(T)}\partial(c\diamond T')\\
 & = & \partial(T)\diamond c\diamond T'+(-1)^{l(T)}c\diamond\partial(T')=\partial(T)\circ T'+(-1)^{l(T)}\circ\partial(T').
\end{eqnarray*}
Next, to see that $L$ is a Maurer-Cartan element, one has to show that, for each $k$, 
\[
 \partial(L_{k})=\sum_{j=1}^{k-1}(-1)^{j+1}L_{j}\circ L_{k-j}+\sum_{j=1}^{k-1}(-1)^{j}d_{j}(L_{k-1}),
\]
which follows immediately from: 
\begin{gather*}
 \partial_{j}^{1}L_{k}=b^{\diamond j}\diamond c\diamond b^{\diamond(k-j)}=L_{j}\circ L_{k-j,}\\
 \partial_{j}^{0}L_{k}=b^{\diamond j}\diamond a\diamond b^{\diamond(k-j)}=d_{j}L_{k-1}.
\end{gather*}
For the universality property of $(\CT,L)$ it is enough to remark that every forest $F\in\CT$ can be written uniquely as
\begin{equation}
 F=d_{i_{1}}\dots d_{i_{m}}(L_{k_{1}}\circ\dots\circ L_{k_{s}}),\label{Factorization}
\end{equation}
with $i_{1}>\dots>i_{m}$.

Finally, the statement about the cohomology follows from the identification with the cellular complexes of the cubes.

\section{Tensor products of representations up to homotopy}

The main conclusion of the previous section is that, with $\Omega$ at hand,  representations up to homotopy are characterized by their characteristic maps (see Corollary \ref{cor-rep-omega}). From this point of view, the construction of tensor products of representations up to homotopy amounts to the construction of diagonal maps on the DB-algebra $\Omega$. For each $m>0$, we consider the DB-algebra
\[
 \Omega_{m}=\underbrace{\Omega\boxtimes\ldots\boxtimes\Omega}_{m-\textrm{times}},
\]
as well as the associated complete DGA $\bar{\Omega}_{m}$. In what follows, $e\in \Omega$ stands for the component $L_1$ of the universal Maurer-Cartan element $L$. 

\begin{definition} 
A \textbf{universal Maurer-Cartan element} of length $m$ is any Maurer-Cartan element $\omega$ of $\bar{\Omega}_{m}$ with the property that 
its degree $1$ component is 
\begin{equation}\label{general-MC}
 \omega_1=e^{\boxtimes m}
\end{equation}
We denote by $\mathcal{MC}_{m}$ the set of such Maurer-Cartan elements. 
\end{definition}

Due to the universal property of $\Omega$, elements $\omega\in\mathcal{MC}_{m}$ can be identified with morphisms in $\underline{\mathcal{DB}ar}$
\[
 \Delta_{\omega}:\Omega\rmap\Omega_{m},
\]
such that $\Delta_{\omega}(e)=e^{\boxtimes m}$. This last condition will allow us to recover the usual diagonal tensor product of (strict) representations. 

\begin{definition} 
A universal Maurer-Cartan element $\omega$ is said to be
\begin{itemize}
 \item \textbf{symmetric} if $\hat{\sigma}(\omega)=\omega$ for all $\sigma\in S_{m}$ (for the action of $S_m$, see subsection \ref{The tensor product of DB-algebras}),
 \item \textbf{associative} when $m=2$ and the induced map $\Delta_{\omega}:\Omega\rmap\Omega\boxtimes\Omega$ is coassociative. 
\end{itemize}
\end{definition}

Coming back to representations up to homotopy, it is now clear that universal Maurer-Cartan elements induce tensor product operations. For instance, using the DB-morphisms $\Delta_{\omega}$ and $m_{E,F}:\mathcal{A}_{E}\boxtimes\mathcal{A}_{F}\rmap\mathcal{A}_{E\otimes F}$ of Proposition \ref{mEF}, we now define:

\begin{definition} 
Given $\omega\in\mathcal{MC}_{2}$ and two representations up to homotopy $E$ and $F$, we define $E\otimes_{\omega}F$ as the representation up to homotopy with the characteristic map 
\[
 k_{E\otimes_{\omega}F}=m_{E,F}\circ(k_{E}\boxtimes k_{F})\circ\Delta_{\omega}.
\]
\end{definition}

From the second part of Proposition \ref{rep-MC} and the naturality of the construction we deduce the following:

\begin{corollary} 
The operations $\otimes_{\omega}$ have the following properties: 
\begin{enumerate}
 \item [(1)] Any strong gauge equivalence between $\omega,\omega'\in\mathcal{MC}_{2}$ induces a strong isomorphism between $E\otimes_{\omega}F$ and $E\otimes_{\omega'}F$ (see Definitions \ref{def-strong} and \ref{strong-appendix}). 
 \item [(2)] If $\omega$ is associative or symmetric, then so is the operation $\otimes_{\omega}$. 
\end{enumerate}
\end{corollary}

Similarly, any universal Maurer-Cartan element $\omega$ of length $m$ induces a tensor product operation on $m$-arguments $\otimes_{\omega}(E_{1},\ldots,E_{m})$, defined by 
\[ 
 k_{\otimes_{\omega}(E_{1},\ldots,E_{m})}=m_{E_{1},\ldots,E_{m}}\circ(k_{E_{1}}\boxtimes\ldots\boxtimes k_{E_{m}})\circ\Delta_{\omega}.
\]
As before, a gauge equivalence between Maurer-Cartan elements induces a strong isomorphism between the corresponding tensor products. Also, if $\omega$ is symmetric, so is the associated tensor product.

\begin{corollary} 
A symmetric universal Maurer-Cartan element $\omega$ of length $m$  induces a symmetric power operation $E\mapsto S^{m}E$ on representations up to homotopy. 
\end{corollary}

In this section we study the universal Maurer-Cartan elements. First of all, we 
clarify their existence and uniqueness.

\begin{theorem} \label{existence-diagonal} 
For each $m>0$, we have the following:
\begin{enumerate}
 \item Symmetric universal Maurer-Cartan elements of length $m$ exist. 
 \item Any two universal Maurer-Cartan elements of length $m$ (symmetric or not) are strongly gauge equivalent, and any two gauge equivalences are homotopic. 
\end{enumerate}
\end{theorem}

Hence the resulting tensor product and symmetric power operations are uniquely defined up to strong isomorphisms. 
However, as the next theorem shows, the tensor product operation does not posses all the properties one would hope for. Namely:

\begin{corollary}\label{diagonal-case-2} 
For $m=2$, 
\begin{enumerate}
 \item There exist universal Maurer-Cartan elements that are associative. 
 \item There is no universal Maurer-Cartan element that is both associative and symmetric. 
\end{enumerate}
\end{corollary}

In this section we also discuss a more special class of universal Maurer-Cartan elements, called rigid, which behave well with respect to the additional product $\diamond$ on $\Omega$.
One advantage of the rigid Maurer-Cartan elements is that they can be described completely. More importantly, the resulting tensor products 
preserve the unitality. To describe them, we begin by extending the product $\diamond$ from $\Omega$ to $\Omega_{m}$ by the usual formula
\[
 (a_{1}\boxtimes a_{2})\diamond(b_{1}\boxtimes b_{2})=(-1)^{l(a_{2})l(b_{1})}(a_{1}\diamond b_{1})\boxtimes(a_{2}\diamond b_{2}),
\]
giving $(\Omega_{m},\diamond)$ the structure of a unital algebra with unit $e^{\boxtimes m}$. 

\begin{definition} 
A  universal Maurer-Cartan element  $\omega\in\mathcal{MC}_{m}$ is called \textbf{rigid} if its characteristic map $\Delta_{\omega}:\Omega\rightarrow\Omega_{m}$ is a map of unital algebras with respect to the product $\diamond$.
\end{definition}

Here are the main properties of the rigid Maurer-Cartan elements.

\begin{theorem}\label{extending second components}
For each $m\geq2$, 
\begin{enumerate}
\item[1.] The set of all rigid 
Maurer-Cartan elements  $\omega\in\mathcal{MC}_{m}$ is in one-to-one correspondence with the set $\mathcal{R}_m$ of elements $x \in\Omega_{m}^{2}(-1)$ that satisfy the following equation
\begin{eqnarray}\label{PathEquation}
 \partial(x) & = & c^{\boxtimes m}-a^{\boxtimes m}.
\end{eqnarray}
The correspondence is characterized by
\begin{eqnarray}
 \omega & = & e^{\boxtimes m}+\sum_{k\geq1}x^{\diamond k}\label{eq:SerreElement}
\end{eqnarray}
\item[2.] There exists a unique  symmetric rigid Maurer-Cartan element in $\mathcal{MC}_{m}$. 
\item[3.] If $\omega\in\mathcal{MC}_{m}$ is rigid, then the induced tensor product $\otimes_{\omega}$ of unital representations up to homotopy is unital.
\end{enumerate}
\end{theorem}

In the case $m= 2$ we deduce the following.

\begin{corollary} \label{space of diagonals} 
The set of rigid Maurer-Cartan elements in $\mathcal{MC}_{2}$ coincides with the one-parameter family
$\{\omega_t\}_{t\in \mathbb{R}}$ given by
\begin{eqnarray*}
 \omega_{t} & = & e\boxtimes e+\sum_{k\geq1}(B+At)^{\diamond k},
\end{eqnarray*}
where
\[
 B=b\boxtimes a+c\boxtimes b,\ \  A= b\boxtimes c+ a\boxtimes b - b\boxtimes a- c\boxtimes b.
\]
Among these, $\omega_{\frac{1}{2}}$ is the only symmetric one, and $\omega_{0}$ and $\omega_{1}$ are the only associative ones.
\end{corollary}

\begin{example}
For $\omega_{0}$, the second component  is $B$ and its third component is $B\diamond B$. In terms of trees it is given by
\begin{gather*}
 \SquareVertexD\;\otimes\;\FSquare\;-\;\SquareEdgeA\;\otimes\;\SquareEdgeD\;+\qquad\qquad\qquad\\
 \\\qquad\qquad\qquad+\quad\SquareEdgeB\;\otimes\;\SquareEdgeC\;+\;\FSquare\;\otimes\;\SquareVertexA\;.
\end{gather*}
It yields, for the third component of the tensor product of the representations up to homotopy $(E, R_{k})$ and $(F, R_{k})$, the expression
\begin{eqnarray*}
 R_{3}(g_{1},g_{2},g_{3}) & = & \Big(R_{1}(g_{1})\circ R_{1}(g_{2})\circ R_{1}(g_{3})\Big)\otimes R_{3} (g_{1},g_{2},g_{3})-\\
 &  & \quad-\Big(R_{1}(g_{1})\circ R_{2}(g_{2},g_{3})\Big)\otimes\Big(R_{2}(g_{1},g_{2}g_{3})\Big)\\
 &  & \quad\quad+\Big(R_{2}(g_{1},g_{2})\circ R_{1}(g_{3})\Big)\otimes R_{2}(g_{1}g_{2},g_{3})\\
 &  & \quad\quad\quad\quad+R_{3}(g_{1},g_{2},g_{3})\otimes R_{1}(g_{1}g_{2}g_{3}).
\end{eqnarray*}
\end{example}

\begin{remark} 
Given $\omega\in \mathcal{MC}_2$,  we can interpret the restriction of the characteristic map $\Delta_\omega$ to $\Omega^{n+1}$ as a diagonal on the cubical complex of the $n$-cube (see paragraph \ref{sub: Cube Description }). If we take the rigid and associative universal Maurer-Cartan element $\omega_0$, the resulting (associative) coincides with the Serre diagonal present in any cubical complex (\cite{Serre}). 
\end{remark}

The rest of this section is devoted to the proofs of the theorems stated above.

\subsection{Proof of Theorem \ref{existence-diagonal}}\label{proof1}

That universal Maurer-Cartan elements exist can be derived abstractly using Proposition \ref{MC1} from the appendix  (for $r= 3$). However, the existence will also follow
from the results on rigid Maurer-Cartan elements. Note also that starting with any Maurer-Cartan element $\omega\in\mathcal{MC}_m$, one can produce a
symmetric one by averaging:
 \[
 \textrm{Av}(\omega)=\frac{1}{m!}\sum_{\sigma\in S_{m}}\hat{\sigma}(\omega).
\]
For the second part of the theorem we use Proposition \ref{MC2} of the appendix applied to $r= 2$
and to the complete DG-algebra associated to $\Omega_m$. 
Denoting by $\ldots \subset F_2\subset F_1\subset F_0$  the associated filtration of $\bar{\Omega}_m$, we have $F_{k}/F_{k+1}= \Omega^{k}(\bullet)$.
Note that any universal Maurer-Cartan element is congruent to $e^{\boxtimes m}$ modulo $F_2$ and the induced differential
on $F_2/F_3$ becomes $\partial$. On the other hand, the cochain complex $(\Omega_{m}^{k}(\bullet), \partial)$ is the the tensor product
of the $m$ copies of the complex $(\Omega^{k}(\bullet), \partial)$.  Hence, from the last part of Theorem \ref{MC-exists},
it has trivial cohomology in non-zero degrees. In conclusion, for 
the cohomology of $F_k/F_{k+1}$ indexed by the total degree (as needed in the theorem of the appendix), we obtain
\[ H^{i}(F_k/F_{k+1})= 0\ \ \forall \ i\neq k\]
and we can apply Proposition \ref{MC2}.

\subsection{Proof of Theorem \ref{extending second components}}

For the first part of the theorem, let $\Delta_{\omega}:\Omega\rightarrow\Omega_{m}$ be the map associated to a rigid element $\omega\in\mathcal{MC}_{m}$. We verify immediately that $x=\Delta_{\omega}(b)$ satisfies Equation (\ref{PathEquation}) by applying $\partial$ on both sides. For the converse, we need to show that given $x$, there is a unique way to extend it to a rigid Maurer-Cartan element $\omega$. Recall that the $k^{th}$ component $\omega_{k}$ of $\omega$ is $\Delta_{\omega}(L_{k})$. Since $L_{k}=b^{\diamond(k-1)}$ and $\Delta_{\omega}$ is required to respect $\diamond$, we conclude that if $\omega$ exists, then $\omega_{k}=x^{\diamond(k-1)}$ for $k\geq2$ and $\omega_{1}=e^{\boxtimes m}$. Now we only need to prove that such $\omega$ is a Maurer-Cartan element.
Note that for $x\in\Omega_{m}^{k(x)}$ and $y\in\Omega_{m}^{k(y)}$ with $k(x)>0$ and $k(y)>0$, the following identities hold:
\begin{eqnarray}
 x\circ y & = & x\diamond c^{\boxtimes m}\diamond y,\nonumber\\
 d_{k(x)}(x\diamond y) & = & x\diamond a^{\boxtimes m}\diamond y.\nonumber
\end{eqnarray}
Using also that $\partial$ is a derivation with respect to $\diamond$, we have
\begin{eqnarray*}
 \partial(x^{\diamond k-1}) & = & \sum_{i=1}^{k-1}(-1)^{i+1}
      \big(x^{\diamond(i-1)}\diamond(c^{\boxtimes m}-a^{\boxtimes m})\diamond x^{\diamond(k-i-1)}\big),\\
      & = & \sum_{i=1}^{k-1}(-1)^{i+1}\big(x^{\diamond(i-1)}\circ x^{\diamond(k-i-1)}-d_{i}x^{\diamond(k-2)}\big),
\end{eqnarray*} 
hence $\omega$ is a Maurer-Cartan element.

For the existence of the second part of the theorem, note that the symmetric element
\[
 x_{m}=\frac{1}{m!}\sum_{\sigma\in S_{m}}\sum_{j=0}^{m-1}\hat{\sigma}\,\,
\left(\underbrace{c\boxtimes\dots\boxtimes c}_{j\text{ times }} \boxtimes b \boxtimes \underbrace{a\boxtimes\dots\boxtimes a}_{m- j- 1\text{ times }} \right)
\]
belongs to $\mathcal{R}_m$. This follows by direct computation (recall that $\partial(b)= c-a$ and $\partial$ kills $a$ and $c$). For the uniqueness, note first that the elements of the form
\[ X_j=c^{\boxtimes j}\boxtimes b\boxtimes a^{\boxtimes m-j-1}, \ \ 0\leq j\leq m-1 \] 
span a vector space consisting of representatives of the orbits the action of the permutation group on $\Omega_{m}^{2}(-1)$. 
Since $\partial$ commutes with the action of the permutation group, averaging gives a one-to-one correspondence between symmetric solutions of the equation (\ref{PathEquation}) a
and solutions of type $X =\sum_i a_i X_i$ of the same equation, with $a_i$-some coefficients. It is now easy to see that the resulting
equation on the $a_i$'s has the unique solution $a_i= 1$.

The last part of the theorem requires a more conceptual understanding of the unitality of representations up to homotopy
and is postponed to the final subsection of this section.

\subsection{The case $m= 2$: the proof of Corollary \ref{diagonal-case-2}
and of Corollary \ref{space of diagonals} }

We start with the proof of Corollary \ref{space of diagonals}.
In view of Theorem \ref{extending second components}, in order to prove that $\omega_t$ is a rigid Maurer-Cartan element, it is enough to show that $ B+At\in \mathcal{R}_2$.
A simple computation shows that:
\begin{eqnarray*}
\partial(B)& =&  c^{\boxtimes m}-a^{\boxtimes m}\\
\partial{A}&=&0.
\end{eqnarray*}
For the uniqueness note that, in general, $\mathcal{R}_m$ is an affine space with underlying vector space consisting of $\partial$-cocycles in $\Omega_{m}^{2}(-1)$. 
Due to the identification with the cellular complex of the cubes, this vector space coincides with $Z^{1}(C(I^m))$. When $m= 2$, this is easily seen to be 1-dimensional,
hence the family $\{\omega_t\}$ exhausts all the rigid elements. 

For the last part of the corollary, 
let $\Delta_{t}:\Omega\rightarrow\Omega\boxtimes\Omega$ be the characteristic map of the rigid element $\omega_{t}\in\mathcal{MC}_{2}$. For each $t\in\R$, $\omega_{t}$ produces an associative tensor product of representations up to homotopy if and only if $\Delta_{t}$ is coassociative. Since $\Delta_{t}$ respects $\diamond$, we only need to check coassociativity on the generators $a,b$ $c$ and $e$. This always holds for $a,c$ and $e$. A direct computation shows that
\begin{eqnarray*} 
 (\Delta_{t}\boxtimes\uid)(B+At) & = & (\uid\boxtimes\Delta_{t})\circ(B+At)
\end{eqnarray*}
if and only if $t\in\{0,1\}$. A similar argument shows that $\Delta_{t}$ is symmetric if and only if $t=\frac{1}{2}$.

Turning to Corollary \ref{diagonal-case-2}, we are left with proving the last part. Assume that $\omega\in \mathcal{MC}_{2}$ is
symmetric and we show that it cannot be associative. We claim that $\Delta_{\omega}(b)$ must belong to $\mathcal{R}_2$. Indeed, since $\Delta_{\Omega}$ is a DB-morphism 
and $\partial(b)= c-a$, $c= e^2$, $a= d_1(e)$, we have
\[ \partial(\Delta_{\omega}(b))= \Delta_{\omega}(c)- \Delta_{\omega}(a)= \Delta_{\omega}(e)^2- d_1(\Delta_{\omega}(e))= c^{\boxtimes 2}-a^{\boxtimes 2}.\]
Hence, from the uniqueness of symmetric elements of $\mathcal{R}_m$, $\Delta_{\omega}(b)$ must coincide with $B+\frac12 A$. But, a simple computation similar to the one above shows that 
\[ (\Delta_{\omega}\boxtimes\uid)(B+\frac12 A)\neq (\uid\boxtimes\Delta_{\omega})(B+\frac12 A),\]
hence $\omega$ cannot be associative.

\subsection{End of proof of Theorem \ref{extending second components}: unitality}\label{sec:unitality}

In this paragraph, we look at tensor products of unital representations up to homotopy, proving in particular the last part of Theorem \ref{extending second components}.
We start by expressing the unitality in terms of the characteristic map $k_E:\Omega\rightarrow\mathcal{A}_E$ of the representation up to homotopy. This brings us to the question of the unitality of the DB-algebras $\Omega$ and $\mathcal{A}_E$ themselves. So far, $\Omega$ has been nonunital ($\Omega^k$ is nontrivial only for $k\geq 1$). On the other hand, $\mathcal{A}_E$ does have a unit (the vector bundle identity map $\uid_E\in \mathcal{A}_E^0$)- but we did not use it so far. 
Unital representations up to homotopy force us to consider unital DB-algebras: throughout this section, we regard $\mathcal{A}_E$ as a unital DB-algebra with unit $\uid_E$, and we formally adjoint a unit to $\Omega$ in degree $0$, which we denote by $1$ and interpret as the empty tree. The characteristic map becomes unital by imposing
$$k_E(1) = \uid_E.$$

\begin{remark}
The set of Maurer-Cartan elements  $\mathcal{MC}_1(\mathcal{\bar A})$  are not concerned by the unitality of the  DB-algebra $\mathcal{A}$,  since we require these elements to have no component in degree zero.  For instance, the universal element $L=L_1+L_2+L_3+\cdots$  in the nonunital $\Omega$ remains the same in its unital version. 
\end{remark}

\begin{definition}
Let  $\mathcal{A}_{E}$ be the DB-algebra associated to a representation up to homotopy $(E,R_{k})$ of a Lie groupoid $G$ over $M$. For each $k\geq 1$, we define the following operators
\[
 s_{i}:A_{E}^{k}\longrightarrow A_{E}^{k-1},\quad i=1,\dots,k
\]
by the formula
\[
 s_{i}(c)(g_{1},\dots,g_{k-1})=c(g_{1},\dots,g_{i-1},x_i,g_{i},\dots,g_{k-1}),
\]
where $x_i = s(g_{i-1})=t(g_i)$.  For $k=1$, we define $s_1(c)$ to be the restriction of $c$ to the unit space $M$ of the groupoid. We extend these operators  to the powers $\mathcal{A}_{E}^{\boxtimes n}$ diagonally. 
\end{definition}

Clearly, $(E,R_k)$ is unital if and only if $s_1(R_1) = \uid_E$ and $ s_{i}(R_{k})  =  0 $ for $k>1$ and $i=,1\dots,k$. Now let us define the corresponding operators at the universal level:

\begin{definition}
For $k>1$, we define the operators 
\[
 s_{i}:\Omega^{k}\rightarrow\Omega^{k-1},\, i=1,\dots,k,
\]
as follows. Let $F\in \Omega^k$ be a short forest.  We denote by $F_{i}$ the forest with $k-1$ leaves obtained from $F$ by deleting its $i$-th leaf. Then,
\begin{enumerate}
\item if $F_i$ is a short forest, we set $s_i(F) =F_i$.
\item if $F_i$ is not a short forest (its $i$-th leaf is now of height $1$), we set
\begin{equation}\nonumber
 s_{i}(F)= \begin{cases}
     F_1\circ F_2,\quad \textrm{if}\quad F=F_1\circ L_1\circ F_2 \textrm{ with }  k(F_1)=i-1,\\
     0, \qquad \quad \textrm{   otherwise}.
    \end{cases}
\end{equation}
\end{enumerate}
For $k=1$, we set $s_1(L_1)= 1$. We extend the $s_{i}$'s diagonally to $\Omega_{m}$.
\end{definition}

\begin{example}
Applying $s_2$ to the short forests
\[
 F=\;\ForestAC\;, \quad G=\;\ForestCC\;\quad H=\;\ForestCB\;,
\]
respectively yields the short forests $c$ for $F$ and $G$, and $0$ for $H$. 
\end{example}

Deleting the $i$-th leaf of a short forest $F\in \Omega^k$ corresponds, at the universal level, to plugging a groupoid unit in the $i$-th argument slot of the operator $k_E(F)$ in $\mathcal{A}_E^k$. In particular, we see that the $s_i$'s commute with the characteristic map $k_E:\Omega\rightarrow \mathcal{A}_E$ of a representation up to homotopy if and only if the representation is unital. Namely, our definitions yield that $s_1\circ k_E(L_1)$ is the restriction of $R_1$ to the unit space and that $k_E\circ s_1(L_1)$ is the identity $\uid_E$. For $n>1$, we have that $s_i\circ k_E(L_n)$ is the operator $R_n$ restricted to the unit space at its $i$-th slot and, on the other hand, that $k_E\circ s_i(L_n) = 0$. 

Now we are ready to characterize the universal Maurer-Cartan elements whose associated tensor products preserve the unitality of the representations.

\begin{definition} 
We will say that a universal Maurer-Cartan element $\omega\in\mathcal{MC}_{m}$ is unital if its characteristic map $\Delta_\omega$ commutes with the $s_{i}$'s.
\end{definition}

\begin{lemma} 
Let $\omega\in\mathcal{MC}_{m}$ be a unital Maurer-Cartan element. Then the tensor product of $m$ unital representations up to homotopy with respect to $\omega$ is also unital. 
\end{lemma}

\begin{proof} 
The characteristic map 
\[
 k_{\otimes_{\omega}(E_{1},\ldots,E_{m})}=m_{E_{1},\ldots,E_{m}}\circ(k_{E_{1}}\boxtimes\ldots\boxtimes k_{E_{m}})\circ\Delta_{\omega},
\]
of the tensor product $\otimes_\omega(E_1,\dots,E_m)$ of the unital representations $E_1,\dots,E_m$  is a composition of three maps, each of which commutes with the $s_i$'s. Hence the composition itself also commutes with the  $s_i$'s and the resulting representation up to homotopy is unital. 
\end{proof}

\begin{proposition} 
For any $x\in\Omega_{m}^{2}(-1)$,
\[
 s_{i}(x^{\diamond k})=0,\quad  k>0,\quad 1\leq i\leq{k+1}.
\]
In particular, if $x\in\mathcal{R}_m$,  then the associated rigid Maurer-Cartan element $\omega$ (of Theorem \ref{extending second components}) is unital.  
\end{proposition} 

\begin{proof} 
For the purpose of this proof, we will say that an element $y\in\Omega_{m}$ is normalized if $s_{i}(y)=0$ for all $i$. The set of elements of the form $e_{1}\boxtimes\cdots\boxtimes e_{m}$, where $e_{i}\in\{a,b,c\}$ and $b$ occurs exactly once, form a basis of $\Omega_{m}^{2}(-1)$. Such a string is normalized if one of its factors is normalized; since $b$ is normalized, all elements $x\in\Omega_{m}^{2}(-1)$ are normalized. Next, suppose that all the powers $x^{\diamond k}$ are normalized for $k<n$. Then one immediately sees that $x^{\diamond n}$ is normalized if and only if 
\begin{eqnarray*}
 s_{2}(x^{\diamond n})=0.
\end{eqnarray*}
On the other hand, since
\[
 s_{2}(x^{\diamond n})=s_{2}(x\diamond x\diamond(x^{\diamond n-2}))=s_{2}(x\diamond x)\diamond(x^{\diamond n-2}),
\]
we conclude that it is enough to prove that $x\diamond x$ is normalized. Let us show that this is the case. First note that $x\diamond x=\frac{1}{2}[x,x]_{\diamond}$, where $[\,,\,]_{\diamond}$ is the graded commutator of the associative product $\diamond$. Therefore, it is enough to show that the graded commutator of two elements of the basis of $\Omega_{m}^{2}(-1)$ mentioned above is always normalized. Let $v=v_{1}\boxtimes\cdots\boxtimes v_{m}$ be a basis vector with the unique occurrence of  $b$ at position $j$, and let $w=w_{1}\boxtimes\cdots\boxtimes w_{m}$ be another basis vector with $b$ at position $l$. If $j=l$, we are done, since the factor $bb$ will appear in both terms of the commutator. If $j<l$, we have that
\[
 [v,w]  =  v_{1}\diamond w_{1}\otimes\cdots\otimes v_{m}\diamond w_{m}-w_{1}\diamond v_{1}
    \boxtimes\cdots\boxtimes w_{m}\diamond v_{m},
\]
the minus sign reflecting that the unique $b$ in $v$ ``passes over'' the unique $b$ in $w$ while carrying out the product. Since both $v$ and $w$ are normalized, we only need to show that $s_{2}$ vanishes on the commutator. This follows from the fact that $s_{2}(x\diamond y)=s_{2}(y\diamond x),$ for $x,y\in\{a,b,c\}$. This last statement can be checked by direct inspection. 
\end{proof}

\section{Tensor products of morphisms}

In the previous section, we have introduced the sets $\mathcal{MC}_{m}$ of universal Maurer-Cartan elements, and we have shown that any $\omega\in\mathcal{MC}_{m}$ induces a tensor product operation $\otimes_{\omega}$ on the objects of $\Rep^\infty(G)$. In this section, we discuss tensor product operations for the morphisms. We will proceed in a way that is completely similar to the tensor product of representations up to homotopy:
\begin{itemize}
\item introduce the notion of a DB-module over a DB-algebra. 
\item describe morphisms between representations up to homotopy in terms
of Maurer-Cartan morphisms. 
\item introduce the Maurer-Cartan module $\mathcal{T}$ (the analogue of $\Omega$). 
\item study universal Maurer-Cartan morphism $x$ (between two universal Maurer-Cartan elements
$\omega$ and $\eta$): existence and uniqueness. 
\item show that any such $x$ induces a tensor product
operation $\otimes_{x}$ on morphisms between representations up to homotopy. 
\end{itemize}
Moreover, we check that the basic properties of the resulting tensor products (e.g. associativity)
hold up to homotopy. The main conclusion will be that the homotopy category $\mathcal{D}(G)$ has a monoidal structure 
uniquely defined up to natural isomorphism.

\subsection{Universal Maurer-Cartan morphisms}

We start with the notion of a DB-module. 

\begin{definition}
Given a DB-algebra $\mathcal{A}$, a left \textbf{DB-module} $\mathcal{E}$ over $\mathcal{A}$, or simply a left $\mathcal{A}$-module, is a bigraded vector space 
\[
 \mathcal{E}=\bigoplus_{l\in\mathbb{Z},k\geq0}\mathcal{E}^{k}(l),
\]
together with a differential $\partial$ and operations $d_{i}$, as in the definition of DB-algebras, and an operation 
\[
 \circ:\mathcal{A}^{k}(l)\otimes\mathcal{E}^{k'}(l')\rmap\mathcal{E}^{k+k'}(l+l'),\ (a,x)\mapsto a\circ x.
\]
These are required to satisfy the same equations (\ref{T-algebras-Leibniz}) and (\ref{T-algebras-di}) as in Definition \ref{T-algebras}, with $a\in\mathcal{A}$, $b\in\mathcal{E}$.  Similarly, one defines the notion of right DB-module. 
\end{definition}

We will be interested in $\mathcal{A}$-$\mathcal{B}$-bimodules with $\mathcal{A}$ and $\mathcal{B}$ two DB-algebras. There is a version of   the functor 
\[ 
 \overline{K}:\underline{\mathcal{DB}ar}\rmap\overline{DGA}
\]
from the category of $\mathcal{A}$-$\mathcal{B}$-bimodules to the category of complete $\bar{\mathcal{A}}$-$\bar{\mathcal{B}}$-DG-bimodules. Thus, given an $\mathcal{A}$-$\mathcal{B}$-bimodule $\mathcal{E}$, there is a complete $\bar{\mathcal{A}}$-$\bar{\mathcal{B}}$-DG-bimodule $(\bar{\mathcal{E}},d_{\textrm{tot}})$, where the left and right actions are defined using the signed operation $\star$ (see equation \eqref{bullet}). Given two Maurer-Cartan elements $\theta\in MC(\bar{\mathcal{A}})$ and $\omega\in MC(\bar{\mathcal{B}})$, we consider the set of Maurer-Cartan $\bar{\mathcal{E}}$-morphisms (see the appendix): 
\[
 \bar{\mathcal{E}}(\omega,\theta)=\{x\in\ \bar{\mathcal{E}}^{0}:x\star\omega-\theta\star x=d_{\textrm{tot}}(x)\}.
\]

Let $G$ be a Lie groupoid over $M$ and $(E,\partial),(F,\partial)$ two cochain complexes of vector bundles over $M$. There is a $\mathcal{A}_{F}$-$\mathcal{A}_{E}$ bimodule that we will denote by $\mathcal{E}_{E,F}$ and which is defined as follows. In bidegree $(k,l)$: 
\[
 \mathcal{E}_{E,F}^{k}(l):=\Gamma(G_{k},\textrm{Hom}^{l}(s^{*}E,t^{*}F)),
\]
and the structure maps are defined exactly as for $\mathcal{A}_{E}$. We can now characterize morphisms of representations up to homotopy in this language.

\begin{proposition}\label{a homotopy gives a homotopy} 
There is a natural bijective correspondence between morphisms  \[\Phi:E\rmap F\] in $\textrm{Rep}^{\infty}(G)$ and elements 
\[
 x\in\bar{\mathcal{E}}_{E,F}(R_{E},R_{F}).
\]
Moreover, two morphisms $\Phi,\Psi:E\rmap F$ are homotopic if and only if the corresponding elements $x,y\in\bar{\mathcal{E}}_{E,F}(R_{E},R_{F})$ are homotopic in the sense of the appendix. 
\end{proposition}

\begin{proof}
The statements follow from comparing the definitions in the appendix with equation (\ref{equations for a map}) in the definition of morphism, and equation (\ref{equations for a homotopy}) in the definition of homotopy. 
\end{proof}

Let $\mathcal{E}$ be an $\mathcal{A}$-$\mathcal{B}$-bimodule, where $\mathcal{A}$ and $\mathcal{B}$ are DB-algebras endowed with Maurer-Cartan elements $\theta\in MC_{1}(\bar{\mathcal{A}})$ and $\omega\in MC_{1}(\bar{\mathcal{B}})$. Then, using the characteristic maps associated to $\theta$ and $\omega$, $\mathcal{E}$ can be given the structure of an $\Omega$-bimodule by the formulas 
\[
 a\circ x:=k_{\theta}(a)\circ x,\ x\circ b:=x\circ k_{\omega}(b),\ \ a,b\in\Omega,x\in\mathcal{E}.
\]
This $\Omega$-bimodule will be denoted by $\mathcal{E}_{\omega,\theta}$ and the associated DG module by $\bar{\mathcal{E}}_{\omega,\theta}$.

For an $\Omega$-bimodule $\mathcal{S}$, we denote by 
\[ 
 \textrm{Hom}_{\underline{\mathcal{DB}ar}}(\mathcal{S},\mathcal{E}_{\omega,\theta})
\]
the space of morphisms of $\Omega$-bimodules. Given $S\in\overline{\mathcal{S}}(L,L)$, where $L$ is the universal Maurer-Cartan element of $\Omega$, there is an induced map 
\begin{equation}
 \textrm{Hom}_{\underline{\mathcal{DB}ar}}(\mathcal{S},\mathcal{E}_{\omega,\theta})\rmap\bar{\mathcal{E}}(\omega,\theta),\ f\mapsto f(S).
\end{equation}

\begin{definition} 
A \textbf{universal Maurer-Cartan module} is an $\Omega$-bimodule $\mathcal{T}$, together with an element $T\in\mathcal{T}(L,L)$, with the property that 
\begin{equation}\label{univ-module}
 \textrm{Hom}_{\underline{\mathcal{DB}ar}}(\mathcal{T},\mathcal{E}_{\omega,\theta})
   \rmap\bar{\mathcal{E}}(\omega,\theta),\ f\mapsto f(T)
\end{equation}
is bijective for all $(\mathcal{A},\theta,\mathcal{E},\mathcal{B},\omega)$ as above. Given $x\in\overline{\mathcal{E}}(\omega,\eta)$, the associated map will be denoted by 
\[
	k_{x}:\mathcal{T}\rmap\mathcal{E}
\]
and will be called the \textbf{characteristic map} of $x$.
\end{definition}

\begin{theorem} 
The Maurer-Cartan DB-module $\mathcal{T}$ exists and is unique up to isomorphism. Moreover, for each $k$, $H^{m}(\mathcal{T}^{k}(\bullet),\partial)=0$
for all $m\neq0$. 
\end{theorem}

\begin{proof} 
The uniqueness follows from the universal property. For the existence part, we will construct $\mathcal{T}$ explicitly. As a vector space, $\mathcal{T}$ is spanned by expressions of type $(A,X,B)$, where each $A$ and $B$ are short forests in $\RS\coprod\{\emptyset\}$, and $X$ belongs to the space of short trees $\RT\coprod\{\emptyset\}$. Note that $A$, $B$ and $X$ may be the empty tree, which we denote by $1$. The bimodule structure is described by the following natural formulas: 
\[
	C\circ(A,X,B)=(C\circ A,X,B),\text{ }\quad(A,X,B)\circ C=(A,X,B\circ C).
\]
We introduce the bigrading on $\mathcal{T}$ by
\[
	k(A,X,B)=k(A)+k(X)+k(B),\ l(A,X,B)=l(A)+(l(X)-1)+l(B),
\]
where we put $k(1)=l(1)=0$. The operators $d_{i}$ are defined exactly as in the case of $\Omega$. Finally, the differential $\partial$ is defined as follows. Denote by $T_{n}$ the element $(1,L_{n},1)\in\mathcal{T}$ and $T_{0}=(1,1,1)$. We set 
\[
	\partial(T_{n})=\sum_{i=0}^{n-1}(-1)^{i}T_{i}\circ L_{n-i}-\sum_{i=1}^{n}L_{i}\circ T_{n-i}+\sum_{i=1}^{n-1}(-1)^{i+1}d_{i}(T_{n-1}),
\]
and extend $\partial$ by forcing it to be a derivation and to commute with the operators $d_{i}$. In order to prove that $\partial^{2}=0$, it is enough to show that $\partial^{2}(T_{n})=0$ for all $n$, and this can be checked by a simple computation. The universal Maurer-Cartan morphism in $\mathcal{T}$ is
\[
	T=\left(1,1,1\right)+\left(1,\;\uno\;,1\right)+\left(1,\;\dos\;,1\right)+\left(1,\;\tres\;,1\right)+\cdots.\
\]
By construction, $\mathcal{T}$ satisfies the universal property. Finally, the statement about the cohomology is analogous to that of $\Omega$.
\end{proof}

Next we consider the $\Omega_{m}$-bimodule
\[
	\mathcal{T}_{m}:=\underbrace{\mathcal{T}\boxtimes\ldots\boxtimes\mathcal{T}}_{m\ \textrm{times}}.
\]

\begin{definition}
Given $\omega, \theta\in\mathcal{MC}_{m}$, a \textbf{universal Maurer-Cartan morphism} from $\omega$ to $\theta$ is any Maurer-Cartan morphism 
$x\in \bar{\mathcal{T}}_m(\omega, \theta)$ with the property that its degree $0$ component is 
\[ x_{0} =\underbrace{T_{0}\boxtimes\ldots\boxtimes T_{0}}_{m\ \textrm{times}}.\]
We denote by $\mathcal{MC}_{m}(\omega, \theta)$ the set of such elements. 
\end{definition}

Because of the universal property of $\mathcal{T}$, an element $x\in\mathcal{MC}_{m}(\omega,\theta)$ may be interpreted as a map of $\Omega$-bimodules
\begin{eqnarray*}
	\Delta_{x}:\mathcal{T} & \longrightarrow & \mathcal{T}_{m, \omega,\theta}.
\end{eqnarray*}
 
\begin{definition} 
Let $\mathcal{E},\mathcal{E'}$ be $\Omega$-bimodules and $\phi,\phi':\mathcal{E}\rightarrow\mathcal{E'}$ morphisms. We say that $\phi$ and $\phi'$ are homotopic if there exists a degree $-1$ linear map $h$ that commutes with the $\Omega$ action and the $d_{i}$ operators such that $hd+dh=\phi-\phi'$. In this case, $h$ is called a homotopy between $\phi$ and $\phi'$. 
\end{definition}

\begin{lemma}\label{lemma homotopy is homotopy} 
Let $\mathcal{A},\mathcal{B}$ be DB algebras and $\theta,\omega$ Maurer-Cartan elements of $\mathcal{A}$ and $\mathcal{B}$, respectively. Suppose also that $\mathcal{E}$ is an $\mathcal{A}$-$\mathcal{B}$-bimodule and that $x$ and $y$ belong to $\overline{\mathcal{E}}(\omega,\theta)$. 
Then homotopies between $x$ and $y$ correspond naturally to homotopies between the characteristic maps $k_{x}:\mathcal{T}\rightarrow\mathcal{E}_{\omega,\theta}$ and $k_{y}:\mathcal{T}\rightarrow\mathcal{E}_{\omega,\theta}$. 
\end{lemma}

\begin{proof} 
The correspondence sends a homotopy $h:\mathcal{T}\rightarrow\mathcal{E}_{\omega,\theta}$ to $h(T)\in\overline{\mathcal{E}}_{\omega,\theta}$. Clearly, since $h$ commutes with the $\Omega$ action and the $d_{i}$ operators, it is determined by the value of $h(T)$. On the other hand, the equation $hd+dh=\phi-\phi'$ corresponds precisely to the equation
\[ 
	x-y=d_{tot}h(T)+h(T)\star\omega+\theta\star h(T).
\]
\end{proof}

As before, Maurer-Cartan morphisms induce tensor product operations between morphisms of representations up to homotopy. Given $x\in\mathcal{T}_{m}(\omega,\theta)$ and morphisms between representations up to homotopy $\Phi_{i}:E_{i}\rmap F_{i}$, $1\leq i\leq m$, we define a morphism 
\[
	\Phi=\otimes_{x}(\Phi_{1},\ldots,\Phi_{m}):\otimes_{\omega}(E_{1},\ldots,E_{m})\rmap\otimes_{\theta}(F_{1},\ldots,F_{m})
\]
by specifying its characteristic map: 
\[
	k_{\Phi}:=m_{\boxtimes}\circ(k_{\Phi_{1}}\boxtimes\ldots\boxtimes k_{\Phi_{m}})\circ\Delta_{x}:\mathcal{T}\rmap\mathcal{E}_{E,F},
\]
where $E=\otimes_{\omega}(E_{1},\ldots,E_{m}),F=\otimes_{\theta}(F_{1},\ldots,F_{m})$, and 
\[
	m_{\boxtimes}:\mathcal{E}_{E_{1},F_{1}}\boxtimes\ldots\boxtimes\mathcal{E}_{E_{m},F_{m}}  \longrightarrow  \mathcal{E}_{E,F}
\]
is defined, as in Proposition \ref{mEF}, by:
\[
	m_{\boxtimes}(\phi_{1}\boxtimes\cdots\boxtimes\phi_{m})(g_{1},\dots,g_{k})  =  \phi_{_{1}}(g_{1},\dots,g_{k})\otimes\cdots
				\otimes\phi_{m}(g_{1},\dots,g_{k}).
\]

\begin{theorem}\label{theorem tensor of morphisms} 
Let $\omega,\theta\ $ be elements of $\mathcal{MC}_{m}$ and $\Phi_{i}:E_{i}\rmap F_{i}$, for $1\leq i\leq m$, morphisms of representations up to homotopy,
$E=\otimes_{\omega}(E_{1},\ldots,E_{m}),F=\otimes_{\theta}(F_{1},\ldots,F_{m})$. Then: 

\begin{enumerate}
	\item $\mathcal{MC}_{m}(\omega,\theta)$ is nonempty. Moreover, every two elements in $\mathcal{MC}_{m}(\omega,\theta)$ are homotopic in the sense of the appendix. 

	\item Any homotopy between $x,y\in\mathcal{MC}_{m}(\omega,\theta)$ induces a homotopy between the morphisms $\otimes_{x}(\Phi_{1},\ldots,\Phi_{m})$ and $\otimes_{y}(\Phi_{1},\ldots,\Phi_{m})$. 

	\item For any $x\in\mathcal{MC}_{m}(\omega,\theta)$, the zeroth component of $\otimes_{x}(\Phi_{1},\ldots,\Phi_{m})$ is the tensor product of the zeroth components of the $\Phi_{i}$'s. Also, if $\Phi_{i}$ are strict morphisms, then so is $\otimes_{x}(\Phi_{1},\ldots,\Phi_{m})$. 

	\item Symmetric Maurer-Cartan morphisms exist. Moreover, any symmetric $x$ induces a symmetric power operation on morphisms. 

	\item The tensor product of morphisms is well defined on homotopy classes. Namely, if $\Phi'_{i}$ is homotopic to $\Phi_{i}$ then $\otimes_{x}(\Phi'_{1},\ldots,\Phi'_{m})$ is homotopic to $\otimes_{x}(\Phi_{1},\ldots,\Phi_{m})$.

	\item For any $x\in\mathcal{MC}_{m}(\omega,\omega)$, the tensor product of the identity morphisms $\otimes_{x}(\uid_{E_{1}},\dots,\uid_{E_{n}})$ is homotopic to the identity morphism on $\otimes_{\omega}(E_{1},\dots,E_{n})$. 
\end{enumerate}

\end{theorem} 

\begin{proof} 
The first claim is a direct application of Proposition \ref{MC11} from the appendix. The second claim is a consequence of Proposition \ref{a homotopy gives a homotopy}. The third part follows from the condition 
\[
	x_{0}=T_{0}\boxtimes\ldots\boxtimes T_{0}
\]
for elements of $\mathcal{MC}_{m}(\omega,\theta)$. The statement about the symmetric elements holds, because given a Maurer-Cartan morphism, one can construct a symmetric one by averaging. 

Let us now prove that the tensor product of morphisms is well defined on homotopy classes. Since homotopy is a transitive relation, we can assume that $\phi_{i}=\phi'_{i}$ for $i>1$. Now fix a homotopy $h_{1}$ between the characteristic maps $k_{\phi_{1}}$ and $k_{\phi'_{1}}$. Then 
\[
	m\circ(h_{1}\boxtimes k_{\phi_{2}}\boxtimes\ldots\boxtimes k_{\Phi_{m}})\circ\Delta_{x}:\mathcal{T}\rmap\mathcal{E}_{E,F}
\]
is a homotopy between the characteristic maps of $\otimes_{x}(\Phi'_{1},\ldots,\Phi'_{m})$ and $\otimes_{x}(\Phi_{1},\ldots,\Phi_{m})$. By Lemma \ref{lemma homotopy is homotopy}, we conclude that the two morphisms are homotopic.

We turn now to the last claim. Consider the natural map of $\Omega$-bimodules $\pi:\mathcal{T}\rightarrow\Omega,$ defined on generators by $\pi(T_{0})=1$, and $\pi(T_{n})=0$ for $n\geq1$. In particular, $\pi$ does not vanish only for those triples $(A,X,B)$ such that $X=1$, in which case we obtain
\[
	\pi(A,1,B)  =  A\circ B.
\]
This induces a map of $\Omega$-bimodules
\[
	\pi^{\boxtimes m}:(\mathcal{T}^{\boxtimes m})_{\omega,\omega}  \longrightarrow  \Omega_{m}.
\]
Because we are taking tensor products of identity morphisms, the characteristic map of $\otimes_{x}(\uid_{E_{1}},\dots,\uid_{E_{m}})$ factors as follows:
	\begin{diagram}
		\mathcal{T} & \rTo^{\quad\Delta_x\quad}  & (\mathcal{T}^{\boxtimes m})_{\omega,\omega}    &
		\rTo^{\quad m_\boxtimes\circ(k_{\uid_{E_{1}}}\boxtimes\ldots\boxtimes
		      k_{\uid_{E_{m}}})\quad} & \mathcal{E}_{E,F}  \\
		  & \rdTo_{\pi^{\boxtimes m}\circ \Delta_x} & \dTo_{\pi^{\boxtimes m}}  & \ruTo_{\gamma}&    \\
		  &          & \Omega_m       &
	\end{diagram}
where $\gamma$ is characterized by the commutativity of the diagram. On the other hand, the characteristic map of the identity morphism of $E$ is given by the composition 
	\begin{diagram}
		\mathcal{T} & \rTo^{k_{\uid_E}} & \mathcal{E}_{E,E} \\
		\dTo^{\iota}& \ruTo_{\gamma}    &                   \\
		\Omega_m    &                   &                   \\
	\end{diagram}
where $\iota$ is the map of $\Omega$-bimodules defined on generators as  $\iota(T_{0})=1$, and $\iota(T_{n})=0$ for $n>0$. Thus, it is enough to prove that the maps $\iota$ and $\pi^{\boxtimes m}\circ\Delta_{x}$ are homotopic. In view of the universal property of $\mathcal{T}$ and Lemma  \ref{lemma homotopy is homotopy}, we only need to prove that the Maurer-Cartan morphisms $x_{\iota},x_{\pi^{\boxtimes m}\circ\Delta}\in\bar{\Omega}_{2} (\omega,\omega)$ associated to $\iota$ and $\pi^{\boxtimes m}\circ\Delta_{x}$ are homotopic in the sense of the appendix. For this, we observe that they coincide modulo $F_{1}\bar{\Omega}_{m}$, and since
\[
	H^{-p}(\Omega_{m}^{p}(\bullet),\partial)=0,\quad\forall p\geq1,
\]
we can use Proposition \ref{MC2} to conclude the claim. 
\end{proof}

\subsection{Composition of Maurer-Cartan morphisms}

We will now express the composition of morphisms of representations up to homotopy in terms of the Maurer-Cartan DB-module $\mathcal{T}$ in order to show that the homotopy category $\mathcal{D}(G)$ has a monoidal structure. For this, we need to consider the tensor product of $DB$-modules. Suppose that $\mathcal{A},\mathcal{B},\mathcal{C}$ are $DB$-algebras, $\mathcal{E}$ is an $\mathcal{A}$-$\mathcal{B}$-modules and $\mathcal{E}'$ is
an $\mathcal{B}$-$\mathcal{C}$-bimodule. Then one can construct the $\mathcal{A}$-$\mathcal{C}$ DB-bimodule $\mathcal{E}\otimes_{\mathcal{B}} \mathcal{E'}$ as follows. As a bigraded vector space, 
$\mathcal{E}\otimes_{\mathcal{B}} \mathcal{E'}$ it is the same as the vector space underlying the usual tensor product of graded bimodules over an algebra. The operators $d_i$ are given by the formulas:
\[
	d_{i}(v\otimes w)=\left\{ 
								\begin{array}{ll}
									d_{i}(v)\otimes w & \text{if $i\leq k_{1}$,}\\
									v\otimes d_{i-k_{1}}(w) & \text{if $i>k_{1}$.}
								\end{array}
						\right.
\]

The operation $\partial$ is given by:
\[\partial(v \otimes w) =\partial(v)\otimes w +(-1)^{l(v)} v \otimes \partial(w).\]

This construction has the property that: 
\[
	\overline{\mathcal{E}_{1}\otimes_{\mathcal{B}}\mathcal{E}_{2}}  \cong \overline{\mathcal{E}}_{1}\otimes_{\overline{\mathcal{B}}}\overline{\mathcal{E}}_{2}.
\]

Now let us go back to our initial goal, which is to express the composition of Maurer-Cartan morphisms in terms of $\mathcal{T}$: consider a sequence of Maurer-Cartan morphisms:
\[
	(\bar{\mathcal{A}},\omega)\overset{x_{1}}{\longrightarrow}(\bar{\mathcal{B}},\theta)\overset{x_{2}}{\longrightarrow}(\bar{\mathcal{C}},\nu),
\]
where $x_{1}\in\bar{\mathcal{E}_{1}}(\omega,\theta)$ and $x_{2}\in\bar{\mathcal{E}_{2}}(\theta,\nu)$.

As explained in the appendix, the composition $x_{2}\circ x_{1}$ is defined as the tensor product $x_{1}\otimes_{\bar{\mathcal{B}}}x_{2}$ in the $\bar{\mathcal{A}}$-$\bar{\mathcal{C}}$-DG-bimodule $\bar{\mathcal{E}_{1}}\otimes_{\bar{\mathcal{B}}}\bar{\mathcal{E}}_{2}$. The characteristic map $k_{x_{2}\circ x_{1}}$ of the composition $x_{2}\circ x_{1}$ can be expressed in terms of the characteristic maps $k_{x_{1}}$ and $k_{x_{2}}$ as follows:

\begin{equation}\label{eq:morphism comp}
 \mathcal{T}\overset{\Lambda}{\longrightarrow}\mathcal{T}\otimes_{\Omega}\mathcal{T}
 \overset{k_{x_{1}}\otimes_{\Omega}k_{x_{2}}}{\longrightarrow}(\mathcal{E}_{1})_{\omega,\theta}\otimes_{\Omega}
 (\mathcal{E}_{2})_{\theta,\nu} \longrightarrow(\mathcal{E}_{1}\otimes_{\mathcal{B}}\mathcal{E}_{2})_{\omega,\nu},
\end{equation}
where $\Lambda$ a canonical map of $\Omega$-bimodules that is completely determined by
\[
 \Lambda(T_{n})=\sum_{i=0}^{n}T_{i}\otimes T_{n-i}.
\]
For morphisms of representations up to homotopy, 
\[
 (\bar{\mathcal{A}}_{E_{1}},R_{1})\overset{\phi_{1}}{\longrightarrow}(\bar{\mathcal{A}}_{E_{2}},R_{2})
 \overset{\phi_{2}}{\longrightarrow}(\bar{\mathcal{A}}_{E_{3}},R_{3}),
\]
the morphisms live in the bimodules $\bar{\mathcal{E}}_{E_{1},E_{2}}$ and $\bar{\mathcal{E}}_{E_{2},E_{3}}$. In order to obtain the characteristic map of the composition $k_{\phi_{2}\circ\phi_{1}}$ we need to further compose \eqref{eq:morphism comp} with the canonical morphism
\[
 m_{\otimes}:\mathcal{E}_{E_{1},E_{2}}\otimes_{\Omega}\mathcal{E}_{E_{2},E_{3}}  \longrightarrow  \mathcal{E}_{E_{1},E_{3}},
\]
which is defined as
\[
 m_{\otimes}(\phi\otimes\psi)(g_{1},\dots,g_{n})  =  \psi(g_{1},\dots,g_{k})\circ\phi(g_{k+1},\dots,g_{n}).
\]

Our next goal is to prove that tensor products and compositions of representation up to homotopy are compatible 

\begin{proposition}\label{pro: compatibility}
Suppose we have $\omega_{1},\omega_{2},\omega_{3}\in\mathcal{MC}_{2}$, together with
\[
 x_{1}\in\mathcal{MC}_{2}(\omega_{1},\omega_{2}),\; 
 x_{2}\in\mathcal{MC}_{2}(\omega_{2},\omega_{3}),\; 
 x_{3}\in\mathcal{MC}_{2}(\omega_{1},\omega_{3}).
\]
Then for any representations up to homotopy $E_i$ and $F_i$, for $i=1,2,3$, and any morphisms  $\phi_{j}:E_{j}\rightarrow E_{j+1}$ and $\psi_{j}:F_{j}\rightarrow F_{j+1}$, for $j=1,2$, the composition of the tensor products 
\[
 E_{1}\otimes_{\omega_{1}}F_{1}
  \overset{\phi_{1}\otimes_{x_{1}}\psi_{1}}{\longrightarrow}
 E_{2}\otimes_{\omega_{2}}F_{2}
  \overset{\phi_{2}\otimes_{x_{2}}\psi_{2}}{\longrightarrow}
 E_{3}\otimes_{\omega_{3}}F_{3}
\]
is homotopic to the tensor product of the compositions
\[
 (\phi_{2}\circ\phi_{1})\otimes_{x_{3}}(\psi_{2}\circ\psi_{1}).
\]
\end{proposition}

\begin{proof}
There is a morphism of $\Omega$-bimodules:
 \begin{equation*}
 {u}: (\mathcal{T} \otimes_{\Omega} \mathcal{T} )\boxtimes (\mathcal{T} \otimes_{\Omega} \mathcal{T}) \rightarrow \overline{\mathcal{E}}_{E_1\otimes E_3, F_1 \otimes F_3},
 \end{equation*}
 defined by the formula $u=m\circ (u_E \boxtimes u_F)$ where:
 \begin{equation*}
u_E: \mathcal{T} \otimes_{\Omega} \mathcal{T} \rightarrow  \overline{\mathcal{E}}_{E_1, E_3}
 \end{equation*}
 is the composition of the map:
  \begin{equation*}
k_{\phi_2} \otimes_{\Omega} k_{\phi_1} : \mathcal{T} \otimes_{\Omega} \mathcal{T} \rightarrow  \overline{\mathcal{E}}_{E_2, E_3}\otimes_{\Omega}  \overline{\mathcal{E}}_{E_1, E_2}
 \end{equation*}
 with the natural quotient map:
 \begin{equation*}
   \overline{\mathcal{E}}_{E_2, E_3}\otimes_{\Omega}  \overline{\mathcal{E}}_{E_1, E_2}\rightarrow  \overline{\mathcal{E}}_{E_1, E_3}.
 \end{equation*}
 Similarly,
  \begin{equation*}
u_F: \mathcal{T} \otimes_{\Omega} \mathcal{T} \rightarrow  \overline{\mathcal{E}}_{F_1, F_3}
 \end{equation*}
 is the composition of the map:
  \begin{equation*}
k_{\psi_2} \otimes_{\Omega} k_{\psi_1} : \mathcal{T} \otimes_{\Omega} \mathcal{T} \rightarrow  \overline{\mathcal{E}}_{F_2, F_3}\otimes_{\Omega}  \overline{\mathcal{E}}_{F_1, F_2}
 \end{equation*}
 with the natural quotient map:
 \begin{equation*}
   \overline{\mathcal{E}}_{F_2, F_3}\otimes_{\Omega}  \overline{\mathcal{E}}_{F_1, F_2}\rightarrow  \overline{\mathcal{E}}_{F_1, F_3}.
 \end{equation*}
There is also a canonical morphism of $\Omega$-bimodules 
 \begin{equation*}
 p: (\mathcal{T} \boxtimes \mathcal{T})\otimes_{\Omega}  (\mathcal{T} \boxtimes \mathcal{T}) \rightarrow (\mathcal{T} \otimes_{\Omega} \mathcal{T})\boxtimes  (\mathcal{T} \otimes_{\Omega} \mathcal{T}),
 \end{equation*}
 given by the formula:
 \begin{equation*}
 (a\boxtimes b)\otimes (c\boxtimes d) \mapsto (-1)^{l(b)l(c)}(a\otimes c )\boxtimes (b\otimes d).
 \end{equation*}
A simple computation shows that
 the characteristic map $k: \mathcal{T}\rightarrow \overline{\mathcal{E}}_{E_1\otimes F_1,E_3 \otimes F_3}$ of $(\Phi_2 \otimes \Psi_2) \circ (\Phi_1\otimes \Psi_1)$ is given by the composition
  \begin{equation*}
 k=u\circ p \circ  (\Delta_{x_1} \otimes \Delta_{x_2}) \circ \Lambda:\mathcal{T} \rightarrow \overline{\mathcal{E}}_{E_1\otimes F_1,E_3 \otimes F_3}.
 \end{equation*}

On the other hand, the characteristic map $k': \mathcal{T}\rightarrow \overline{\mathcal{E}}_{E_1\otimes F_1,E_3 \otimes F_3}$ of $(\Phi_2 \circ \Phi_1) \otimes (\Psi_2 \circ \Psi_1)$ is given by the composition
 \begin{equation*}
 k'=u\circ (\Lambda \boxtimes \Lambda) \circ \Delta_{x_3}:\mathcal{T} \rightarrow \overline{\mathcal{E}}_{E_1\otimes F_1,E_3 \otimes F_3}.
 \end{equation*}
 We need to prove that the maps $k$ and $k'$ are homotopic. Clearly, it is enough to prove that the maps:
 \begin{equation*}
 p \circ  (\Delta_{x_1} \otimes \Delta_{x_2}) \circ \Lambda:\mathcal{T} \rightarrow (\mathcal{T} \otimes_{\Omega} \mathcal{T}) \boxtimes  (\mathcal{T} \otimes_{\Omega} \mathcal{T}),
 \end{equation*}
 and
 \begin{equation*}
 (\Lambda \boxtimes \Lambda) \circ \Delta_{x_3}:\mathcal{T} \rightarrow (\mathcal{T} \otimes_{\Omega} \mathcal{T}) \boxtimes  (\mathcal{T} \otimes_{\Omega} \mathcal{T}),
 \end{equation*}
are homotopic. Let us now prove that this last statement is true for any two maps of $\Omega$-bimodules $a,b:\mathcal{T}\rightarrow (\mathcal{T} \otimes_{\Omega} \mathcal{T}) \boxtimes  (\mathcal{T} \otimes_{\Omega} \mathcal{T})$, provided they take the same value on $T_0$, which is satisfied in our case. In order to simplify the notation, let us denote the $\Omega$-bimodule  $(\mathcal{T} \otimes_{\Omega} \mathcal{T}) \boxtimes  (\mathcal{T} \otimes_{\Omega} \mathcal{T})$ by $\mathcal{P}$.  In view of the universal property of $\mathcal{T}$ we can identify $a$ and $b$ with Maurer-Cartan morphisms $\overline{a}, \overline{b} \in \overline{\mathcal{P}}(L,L)$. Also, Lemma \ref{lemma homotopy is homotopy} implies that it is enough to prove that the Maurer-Cartan morphisms $\overline{a}$ and $\overline{b}$ are homotopic. Since $a$ and $b$ coincide when applied
to $T_0$ we know that 
\begin{equation*}
\overline{a} \equiv \overline{b} \, \text{mod} (F_1(\overline{\mathcal{P}})).
\end{equation*}
 By Proposition \ref{MC11} applied to the case $r=1$, we know that it is enough to prove that 
 \begin{equation*}
 H^{0}(F_k(\overline{\mathcal{P}})/ F_{k+1}(\overline{\mathcal{P}}),d_{L,L})=0, \text{ for all } k\geq 	1.
 \end{equation*}
 On the other hand, $F_k(\overline{\mathcal{P}})/ F_{k+1}(\overline{\mathcal{P}})$ is naturally isomorphic to 
 $\mathcal{P}^k$ and, because of the shift in degree, we obtain that:
  \begin{equation*}
 H^{0}(F_k(\overline{\mathcal{P}})/ F_{k+1}(\overline{\mathcal{P}}),d_{L,L})=H^{-k}(\mathcal{P}^k).
 \end{equation*}
 Thus, all we need to prove is that the cohomology of $\mathcal{P}^k$ vanishes in negative degree. For this we observe
 that the complex
 \begin{equation*}
\mathcal{P}^k= ((\mathcal{T} \otimes_{\Omega} \mathcal{T}) \boxtimes  (\mathcal{T} \otimes_{\Omega} \mathcal{T}))^k
 \end{equation*}
 is the tensor product of two copies of the complex $(\mathcal{T} \otimes_{\Omega} \mathcal{T})^k$ and therefore, by K\"unneth's formula,
 it is enough to prove that the cohomology of that last complex vanishes in negative degree. This last claim follows from the 
 fact that the cohomology of $\mathcal{T}$ vanishes in negative degree.

\end{proof}

\subsection{Monoidal structure on $\mathcal{D}(G)$}

In this paragraph, we show that the tensor product operations defined above give the homotopy category $\mathcal{D}(G)$ a monoidal structure.  
The reader is referred to \cite{MacLane} for the basic facts and definitions concerning monoidal categories.

\begin{proposition}
Let us fix a universal Maurer-Cartan element $\omega\in\mathcal{MC}_{2}$ and a universal Maurer-Cartan morphism $x\in\mathcal{MC}_{2}(\omega,\omega)$.  
The corresponding operations of tensor product defines a functor:
\[
	\otimes_{\omega,x}:\mathcal{D}(G)\times \mathcal{D}(G) \longrightarrow \mathcal{D}(G).
\]

Moreover, this functor does not depend on $x$.
\end{proposition}
\begin{proof}
First we observe that Theorem \ref{theorem tensor of morphisms} guaranties that the tensor product operation is well defined on homotopy classes of morphisms and therefore
this map is well defined.
Theorem \ref{theorem tensor of morphisms} and Proposition \ref{pro: compatibility} tell us that
\begin{eqnarray*}
	\uid_{E}\otimes_{x}\uid_{F} 
									& \sim &  \uid_{E\otimes_{\omega}F},\\
	(\phi_{2}\circ\phi_{1})\otimes_{x}(\psi_{2}\circ\phi_{1})  
									& \sim  & (\phi_{2}\otimes_{x}\psi_{2})\circ(\phi_{1}\otimes_{x}\psi_{1}).
\end{eqnarray*}
Observe also that $\otimes_{\omega,x}$ does not depend on $x$ since any two $x,x'\in\mathcal{MC}_{2}(\omega,\omega)$ are homotopic and, thus, the corresponding morphisms coincide in $\mathcal{D}(G)$.
\end{proof}

\begin{proposition} \label{pro: equivalence of bifunctors}
Let $\omega,\omega'\in\mathcal{MC}_{2}$ be two universal Maurer-Cartan elements. Then the bifunctors $\otimes_{\omega}$ and $\otimes_{\omega'}$ on $\mathcal{D}(G)$ are equivalent. Moreover, for any two representations up to homotopy $E_1$ and $E_2$, 
the isomorphisms 
\[
	\uid_{E_{1}}\otimes_{y}\uid_{E_{2}}:E_{1}\otimes_{\omega}E_{2}  \longrightarrow  E_{1}\otimes_{\omega'}E_{2},
\]
represents an isomorphism in $\mathcal{D}(G)$ which is independent of the choice of $y\in\mathcal{MC}_{2}(\omega,\omega')$.
\end{proposition}

\begin{proof}
We need to check that $(\phi_{1}\otimes_{x'}\phi_{2})\circ(\uid_{E_{1}}\otimes_{y}\uid_{E_{2}})$ is homotopic to $(\uid_{F_{1}}\otimes_{y}\uid_{F_{2}})\circ(\phi_{1}\otimes_{x}\phi_{2})$ for any $x\in\mathcal{MC}_{2}(\omega,\omega)$ and for any $x'\in\mathcal{MC}_{2}(\omega',\omega')$ and for any two representation morphisms $\phi_{i}:E_{i}\rightarrow F_{i}$, $i=1,2$. This is guaranteed by Proposition \ref{pro: compatibility}, which tells us that
\[
	(\phi_{1}\otimes_{x'}\phi_{2})\circ(\uid_{E_{1}}\otimes_{y}\uid_{E_{2}})\;
				\sim\;\phi_{1}\otimes_{y}\phi_{2}\;
				\sim\;(\uid_{F_{1}}\otimes_{y}\uid_{F_{2}})\circ(\phi_{1}\otimes_{x}\phi_{2}).
\]
\end{proof}

We now prove that the bifunctor $\otimes_{\omega}$ endows the homotopy category $\mathcal{D}(G)$ with the structure of a monoidal category. 

\begin{theorem} 
For every $\omega\in\mathcal{MC}_{2}$, the functor $\otimes_\omega: \mathcal{D}(G)\times \mathcal{D}(G) \rightarrow  \mathcal{D}(G)$ gives the 
category $\mathcal{D}(G)$ a monoidal structure with unit object given by the trivial representation. Moreover, any two choices of $\omega$ give naturally equivalent monoidal categories. 
\end{theorem} 
\begin{proof}
In order to prove that this functor gives $\mathcal{D}(G)$ a monoidal structure we need to show that there is a unit for the tensor product and that there are
unitors and associators that satisfy  the pentagon and triangle axiom. The general idea of the proof is simple: one first shows that the associators and unitors are given
by universal maps at the level of the DB-algebra $\Omega$. Then the commutativity of the diagrams translates to the statement that certain maps between powers of $\Omega$ are homotopic, which follows from
the fact that the algebras are contractible.

The unit is given as  usual by the trivial representation. Let us first consider the existence of the unitors. We will prove that 
\[ E \otimes 1\cong E;  1\otimes E\cong E,\]
where the maps, which we denote by $\lambda_r$ and $\lambda_l$ respectively,  are given by the obvious identification of vector bundles. 

We denote by $k_E$ the characteristic map of $E$ and consider the map of $DB$-algebras \[\pi: \Omega \boxtimes \Omega \rightarrow \Omega \] defined on generators by setting:
\begin{eqnarray*}
\pi(a\boxtimes b)=
\begin{cases}
a \text{ if } l(b)=0,\\
0 \text{ otherwise.}
\end{cases}
\end{eqnarray*}

Then, the following diagram commutes:

$$
\xymatrix{
 \Omega \boxtimes \Omega \ar[r]^{k_E \boxtimes k_1}\ar[dr]^{\pi}& \mathcal{A}_E \boxtimes \mathcal{A}_1 \ar[r]^{m_E,1}& \mathcal{A}_{E\otimes 1} \ar[r]^\cong&\mathcal{A}_{E}\\
\Omega  \ar[u]^{\Delta_\omega}\ar[r]^{\gamma} & \Omega \ar[urr]^{k_E} & & &
}
$$

Thus, we see that the characteristic map of $E\otimes_\omega 1$ differs from the characteristic map of $E$ by pre-composing with the map $\gamma :\Omega \rightarrow \Omega$. Also, it is clear that $\gamma(L_1)=L_1$.
We claim that $\gamma$ is the identity. By the universal property of $\Omega$ it is enough to prove that $\gamma(L_k)=L_k$. We will prove this assertion inductively. Assume that the statement is true for
$i<k$.  Then:

\[
 \partial(\gamma(L_{k}))=\gamma(\partial(L_k))=\sum_{j=1}^{k-1}(-1)^{j+1}L_{j}\circ L_{k-j}+\sum_{j=1}^{k-1}(-1)^{j}d_{j}(L_{k-1})=\partial(L_k).
\]

On the other hand, since $\gamma$ preserves both the degree and the order, we know that $\gamma(L_k)$ is a multiple of $L_k$. We conclude that they are equal and therefore $E=E\otimes 1$ under the obvious 
identification. By the same argument we know that $E =1 \otimes E$.

Let us now construct the  associators. The characteristic map $\Delta_{\omega}: \Omega \rightarrow \Omega_2$ can be  used to construct two Maurer-Cartan elements $\alpha, \beta \in \mathcal{MC}_3$ with characteristic maps:
\begin{eqnarray*}
\Delta_\alpha=( \Delta_{\omega} \boxtimes \uid)\circ \Delta_{\omega}:\Omega \rightarrow \Omega_3,\\
\Delta_\beta=(\uid \boxtimes \Delta_{\omega})\circ \Delta_{\omega}: \Omega \rightarrow \Omega_3.
\end{eqnarray*}
Theorem \ref{theorem tensor of morphisms} guaranties that there is universal Maurer-Cartan morphism $u$ from $\alpha$ to $\beta$ and that any
two such are homotopic. For any three representations up to homotopy $E_1, E_2, E_3$ of $G$, the gauge equivalence
$u$ induces an isomorphism
$$\hat{u}: (E_1\otimes_\omega E_2)\otimes_\omega E_3 \rightarrow E_1\otimes_\omega (E_2\otimes_\omega E_3), $$
which we will call the associator of the monoidal structure. Note that since any two choices of $u$ are homotopic,
this map is well defined on $\mathcal{D}(G)$. 

Let us now prove that the units are compatible with the associators. We need to prove the triangle axiom, which is the commutativity of the following diagram in $\mathcal{D}(G)$:

$$
\xymatrix{
(E_1 \otimes_\omega 1) \otimes_\omega E_2 \ar[r] ^u  \ar[d]^{\lambda_r \otimes \uid}& E_1 \otimes_\omega(1 \otimes_\omega E_2)\ar[ld]^{\uid \otimes \lambda_l}\\
E_1 \otimes_\omega E_2&\\
}
$$
By the last part of Theorem \ref{theorem tensor of morphisms} we know that $\uid_{E_1} \otimes \uid_{E_2}$ is homotopic to the identity morphism on $E_1 \otimes E_2$.
Therefore, it suffices to prove that the following commutes:
$$
\xymatrix{
(E_1 \otimes_\omega 1) \otimes_\omega E_2 \ar[r] ^u \ar[d]^{=}& E_1 \otimes_\omega(1 \otimes_\omega E_1)\ar[d]^{=}\\
E_1 \otimes_\omega E_2\ar[r]^{\uid_{E_2} \otimes \uid_{E_2}}& E_1 \otimes_\omega E_2\\
}
$$

Let us consider the characteristic map $\Delta_u: \mathcal{T}\rightarrow \mathcal{T}_{3,\alpha, \beta}$ and the map of $\Omega$-bimodules:
$\mu: \mathcal{T}_{3,\alpha, \beta}\rightarrow \mathcal{T}_{2, \omega,\omega}$ defined on generators by the formula:
\begin{eqnarray*}
\mu(a_1\boxtimes a_2\boxtimes a_3)=
\begin{cases}
a_1 \boxtimes a_3 \text{ if } l(a_2)=0,\\
0 \text{ otherwise.}
\end{cases}
\end{eqnarray*}
We will show that $\mu$ is a map of right modules, the other case follows by a symmetric argument.
For $a \in \Omega_3$ and $T\in \mathcal{T}_3$ one easily checks that:

\[\mu (Ta )=\mu(T)\left((\pi \boxtimes \uid)(a)\right), \]
where $\pi:\Omega_2 \rightarrow \Omega$ is the map defined above, for which we proved that $\pi \circ \Delta_\omega =\uid$. Now we take $F \in \Omega$ and compute:
\[\mu (T  F)=\mu (T(\Delta_\omega \boxtimes \uid)(\Delta_\omega)(F))=\mu (T) (\pi \boxtimes \uid)(\Delta_\omega \boxtimes \uid)(\Delta_\omega)(F))=\mu(T)\Delta_\omega(F)=\mu(T)F. \]
Thus, $\mu$ is indeed a map of $\Omega$ bimodules. We now observe that the characteristic map of the morphism 
\[\hat{u}: (E_1\otimes_\omega 1)\otimes_\omega E_2 \rightarrow E_1\otimes_\omega (1\otimes_\omega E_2)\]
factors through the map $\mu: \mathcal{T}_{3,\alpha,\beta}\rightarrow \mathcal{T}_{2,\omega, \omega}$ which implies that the diagram above commutes. We conclude that
the unit is compatible with the associators.

Let us now prove that these associators satisfy the pentagon axiom. We need to prove that for any four representations up to homotopy $E_1,E_2,E_3,E_4$ 
the following composition is the identity.
$$
\xymatrix{
((E_1 \otimes_\omega E_2) \otimes_\omega E_3) \otimes_\omega E_4\ar[d]^{u\otimes \uid}& (E_1 \otimes_\omega E_2) \otimes_\omega (E_3 \otimes_\omega E_4)\ar[l]^{u^{-1}}\\
(E_1 \otimes_\omega (E_2 \otimes_\omega E_3)) \otimes_\omega E_4\ar[d]^{u}& E_1 \otimes_\omega (E_2 \otimes_\omega ((E_3 \otimes_\omega E_4)) \ar[u]^{u^{-1}}\\
E_1 \otimes_\omega ((E_2 \otimes_\omega E_3) \otimes_\omega E_4)\ar[ur]^{\uid \otimes u}
}
$$

Observe that the five ways of putting brackets in the tensor product correspond to five elements $\theta_i \in \mathcal{MC}_4$ which one
can construct from $\omega$. These elements are given by maps $\Delta_{\theta_i}: \Omega \rightarrow \Omega_4$ defined as follows:
\begin{eqnarray*}
&&\Delta_{\theta_1}=(\Delta_\alpha \boxtimes \uid)\circ \Delta_\omega \text{ corresponds to }((E_1 \otimes_\omega E_2) \otimes_\omega E_3) \otimes_\omega E_4,\\
&&\Delta_{\theta_2}=(\Delta_\beta \boxtimes \uid)\circ \Delta_\omega \text{ corresponds to }(E_1 \otimes_\omega (E_2 \otimes_\omega E_3)) \otimes_\omega E_4,\\
&&\Delta_{\theta_3}=(\uid \boxtimes \Delta_\alpha )\circ \Delta_\omega \text{ corresponds to }E_1 \otimes_\omega ((E_2 \otimes_\omega E_3) \otimes_\omega E_4),\\ 
&&\Delta_{\theta_4}=(\uid \boxtimes \Delta_\beta)\circ \Delta_\omega \text{ corresponds to } E_1 \otimes_\omega (E_2 \otimes_\omega ((E_3 \otimes_\omega E_4)),\\
&&\Delta_{\theta_5}=(\Delta_\omega \boxtimes \Delta_\omega)\circ \Delta_\omega \text{ corresponds to }(E_1 \otimes_\omega E_2) \otimes_\omega (E_3 \otimes_\omega E_4). 
\end{eqnarray*}

Here, as before $\Delta_\alpha=(\Delta_\omega \boxtimes \uid) \circ \Delta_\omega$ and $\Delta_\beta=(\uid \boxtimes \Delta_\omega) \circ \Delta_\omega$.
Now, the morphisms that appear in the pentagon axiom are induced by Maurer-Cartan morphisms between the elements $\theta_i$. In order to write them down
we choose specific characteristic maps $\Delta_x:\mathcal{T}\rightarrow \mathcal{T}_{2,\omega,\omega}$,  $\Delta_u:\mathcal{T}\rightarrow \mathcal{T}_{3,\alpha,\beta}$ and
 $\Delta_{u^{-1}}:\mathcal{T}\rightarrow \mathcal{T}_{3,\beta,\alpha}$ which induce the tensor product of morphisms and the associators in $\mathcal{D}(G)$.
 The Maurer-Cartan morphisms $\phi_i: \theta_i \rightarrow \theta_{i+1}$ for $i=1,\dots , 5 \text{ (mod } 5)$ are determined by maps $\Delta_{\phi_i}:\mathcal{T}\rightarrow \mathcal{T}_{4,\theta_{i},\theta_{i+1}}$
 given by the formulas:
 \begin{eqnarray*}
&&\Delta_{\phi_1}=(\Delta_u \boxtimes \uid) \circ \Delta_x,\\
&&\Delta_{\phi_2}=(\uid \boxtimes \Delta_x \boxtimes \uid)\circ \Delta_u,\\
&&\Delta_{\phi_3}=(\uid \boxtimes \Delta_u)\circ \Delta_x,\\
&&\Delta_{\phi_4}=(\uid \boxtimes \uid \boxtimes \Delta_x)\circ \Delta_{u^{-1}},\\
&&\Delta_{\phi_5}=(\Delta_x \boxtimes \uid \boxtimes \uid) \circ \Delta_{u^{-1}}. 
\end{eqnarray*}
 
Let us consider the composition of Maurer-Cartan morphisms:
 \[\phi:=\phi_5 \circ \phi_4 \circ \phi_3 \circ \phi_2\circ  \phi_1   \in \overline{P}(\theta_1, \theta_1),  \]
 where $P$ is the $\Omega$-bimodule defined by:
\[ 
P:=(\mathcal{T}_{4,\theta_5,\theta_1})\otimes_{\Omega} \mathcal{T}_{4,\theta_4,\theta_5}\otimes_{\Omega} (\mathcal{T}_{4,\theta_3,\theta_4}) \otimes_{\Omega} (\mathcal{T}_{4,\theta_2,\theta 3})\otimes_{\Omega} (\mathcal{T}_{4,\theta_1,\theta_2}).\]

The element $\phi$ is related to the maps in the pentagon as follows. There is a natural map of $\Omega$-bimodules
$\sigma:P \rightarrow \mathcal{T}_{4,\theta_1, \theta_1}$ and this defines a universal Maurer-Cartan morphism $\sigma (\phi)\in \mathcal{MC}_4(\theta_1, \theta_1)$.
The composition of all the maps in the diagram is precisely given by:
\[\otimes_{\sigma(\phi)}(\uid_{E_1},\dots, \uid_{E_4}):((E_1 \otimes_\omega E_2) \otimes_\omega E_3) \otimes_\omega E_4   \rightarrow 
((E_1 \otimes_\omega E_2) \otimes_\omega E_3) \otimes_\omega E_4.\]
By the last part of Theorem \ref{theorem tensor of morphisms} we know that this morphism is homotopic to the identity, and therefore equal to the identity in $\mathcal{D}(G)$.

 We conclude that any $\omega$ defines a monoidal structure on $\mathcal{D}(G)$. In a similar manner one can show
that the natural equivalences defined in Proposition \ref{pro: equivalence of bifunctors} are compatible with the associators and therefore are equivalences of monoidal categories.
\end{proof}

\section{Canonical tensor products on morphisms}

In this section, we point out another universal property of the universal Maurer-Cartan module, a property which reveals relationships 
with Hochschild cohomology and non-commutative differential forms. In particular, using the universal derivation, we show that
any universal Maurer-Cartan element $\omega\in \mathcal{MC}_m$ comes together with a canonical (and explicit) universal Maurer-Cartan
endomorphism
\[ x_{\omega}\in  \mathcal{T}_m(\omega, \omega).\]
As an immediate consequence, once a Maurer-Cartan element $\omega \in \mathcal{MC}_m$ is fixed,  
there is a canonical way of taking tensor products of morphisms. 
This is important when one is forced to work with the category $\Rep^{\infty}(G)$ instead of the homotopy category
(e.g. in the search of infinitesimal models for the cohomology of classifying spaces). Note that, at the level
of $\Rep^{\infty}(G)$, the resulting $\otimes_{\omega}$ is a ``functor up to homotopy''.

We begin with the description of the universal derivation.

\begin{proposition} The $\Omega$-bimodule $\mathcal{T}$ admits a unique biderivation $\delta: \Omega \rightarrow \mathcal{T}$ of bidegree $(0, -1)$, which is compatible with the $d_i$'s and which sends $L_n$ to $T_n$. 

Moreover, $\delta$ does not commute with $\partial$, instead, for $A\in \Omega$:
\[ \delta(\partial(A))+ \partial(\delta(A))= T_0 A- A T_0.\]
\end{proposition}

In other words, $\delta: \Omega\rmap \mathcal{T}$
is a linear map which sends elements of bidegree $(k,l)$ into those of bidegree $(k,l-1)$, commutes with the
$d_i$'s and satisfies the derivation condition
\[ \delta(A\circ B)= \delta(A)B+ (-1)^{l(A)}A\circ \delta(B).\]
Since $\Omega$ is generated as a DB-algebra by the $L_n$'s,  the proposition is straightforward. 

\begin{remark} Let us explain how derivations comes into the picture (even before $\mathcal{T}$!),
starting from the notion of morphisms between Maurer-Cartan elements and in what sense $\delta$ is  universal.
Let $\mathcal{A}$ and $\mathcal{B}$ be
two DB-algebras, $\mathcal{E}$ be an $\mathcal{A}$-$\mathcal{B}$-DB-module, 
$\theta\in MC(\bar{\mathcal{A}})$ and $\omega\in MC(\bar{\mathcal{B}})$. Let us try to understand 
elements $x\in \mathcal{E}(\omega, \theta)$ directly in terms of $\Omega$; one would like to re-interpret
the components $x_k$ of $x$ as images of the elements $L_k\in \Omega$ of a certain map
\[ \delta_x: \Omega \rmap \mathcal{E} .\]
The equations that the $x_k$'s must satisfy (involving $\omega$ and $\theta$) indicate that
$\delta_x$ should be required to be a derivation on the bimodule $\mathcal{E}_{\omega, \theta}$.
Adding the extra-condition that $\delta_x$ commutes with the $d_i$'s, $\delta_x$ will be determined uniquely. In turn, 
the fact that $x\in \mathcal{E}(\omega, \theta)$ is
equivalent to the equation 
\[ \delta_x(\partial(A))+ \partial(\delta_x(A))= x_0A- Ax_0\]
for all $A\in \Omega$. Note also that $\delta_x$ does not make use of $x_0$, hence we are really looking at triples
$(\mathcal{E}, \delta, x_0)$ with such properties. Among these,  $(\mathcal{T}, \delta, T_0)$ shows up as the universal 
one. See also Remark \ref{non-commutative} below.
\end{remark}

In order to simplify formulas (and rather intricate signs) we will need some formalism. 
Let us first introduce some terminology.

Consider the category $\underline{\mathbb{VS}}_{B}$ of \textbf{B-vector spaces}, whose objects are collections $V= \{V^k, d_i\}$ consisting of vector spaces $V^{k}$ (one for each integer $k\geq 0$)
and maps $d_i: V^k\rmap V^{k+1}$ for $1\leq i\leq k$ satisfying $d_jd_i= d_i d_{j-1}$ for $i< j$. 
A morphism from $V$ to $W$ consists of families of maps from $V^k$ to $W^k$, commuting with all the operators $d_i$;
such morphisms form the hom-spaces $\textrm{Hom}_{B}(V, W)$. 
As in the case of simplicial vector spaces, one 
can realize $\underline{\mathbb{VS}}_{B}$ as the category of contravariant functors
from a small category $\mathbb{B}$ to the category $\underline{\mathcal{V}}$ of vector spaces.

Associated to $\underline{\mathcal{V}}_{B}$ is the category $\textrm{Gr}(\underline{\mathbb{VS}}_{B})$ of graded objects of $\underline{\mathbb{VS}}_{B}$
and $\textrm{Ch}(\underline{\mathbb{VS}}_{B})$ of cochain complexes in $\underline{\mathbb{VS}}_{B}$. Given  $X$ and $Y$ graded
objects in $\underline{\mathbb{VS}}_{B}$, one defines the graded hom $\underline{\textrm{Hom}}^{*}_{B}(X, Y)$ whose degree $l$-part is
\[ \underline{\textrm{Hom}}^{l}_{B}(X, Y)= \prod_{p} \textrm{Hom}_{B}(X(p), X(p+ l)).\]
When $X$ and $Y$ are complexes in $\underline{\mathbb{VS}}_{B}$, then $\underline{\textrm{Hom}}^{*}_{B}(X, Y)$ has a natural differential:
\[ \partial(f)= \partial\circ f- (-1)^{l(f)} f\circ \partial,\]
where $l(f)$ is the degree of $f$. Note that the internal hom of $\textrm{Ch}(\underline{\mathbb{VS}}_{B})$ is the 
space of zero-cocycles of $\underline{\textrm{Hom}}^{*}_{B}$.

The category $\underline{\mathbb{VS}}_{B}$ comes with a tensor product operation $\otimes$ which makes it into a monoidal category: for $V$ and $W$ in $\underline{\mathbb{VS}}_{B}$,
$V\otimes W$ is defined by
\[ (V\otimes W)^{k}= \bigoplus_{k_1+ k_2= k} V^{k_1}\otimes W^{k_2},\]
with the operators $d_{i}$ given by the formulas:

$$
d_{i}(v\otimes w) =\left\{ \begin{array}{ll}
d_{i}(v)\otimes w& \text{if $i\leq k_1$,}\\
 v\otimes d_{i-k_1}(w)& \text{if $i>k_1$,}\\
\end{array} \right.
$$
where $k_1$ is the degree of $v$. There is an obvious notion of tensor products of morphisms in $\underline{\mathbb{VS}}_{B}$, and the unit is the base field concentrated in degree zero.
This tensor product operation extends to $\textrm{Gr}(\underline{\mathbb{VS}}_{B})$ and $\textrm{Ch}(\underline{\mathbb{VS}}_{B})$  in the standard way: 
\[ (X\otimes Y)(l)= \bigoplus_{l_1+l_2=l} X(l_1)\otimes X(l_2)\]
It also extends to the graded-hom's using the standard sign conventions: 
\[ \otimes:  \underline{\textrm{Hom}}_{B}^l(X, X')\times  \underline{\textrm{Hom}}_{B}^{l'}(Y, Y')\rmap \underline{\textrm{Hom}}_{B}^{l+ l'}(X\otimes Y, X'\otimes Y'), \]
\[ (f\otimes g)(x\otimes y)= (-1)^{l(g) l(x)} f(x)\otimes g(y).\]

Note that a DB-algebra is the same as a monoid in the monoidal category $\textrm{Ch}(\underline{\mathbb{VS}}_{B})$.
With this in mind, there is a B-version of Hochschild cohomology.
Given a DB-algebra $\mathcal{A}$ and an $\mathcal{A}$-bimodule $\mathcal{E}$, we consider the vector spaces
\[ C^{p,l}(\mathcal{A}, \mathcal{E}):= \underline{\textrm{Hom}}^{l}_{B}(\mathcal{A}^{\otimes p},\mathcal{E}),\]
and the space of Hochschild cochains:
\[C^n(\mathcal{A},\mathcal{E})=\bigoplus_{p+l=n}C^{p,l}(\mathcal{A}, \mathcal{E}).\]
We define the differentials by the same formulas as in the case of DG-algebras, but taking into account only the $l$-degree.
More precisely, the horizontal differential 
\begin{equation*}
b: C^{p, l}(\mathcal{A}, \mathcal{E})\rmap C^{p+1, l}(\mathcal{A}, \mathcal{E}),
\end{equation*} 
is given by
\begin{eqnarray*}
 b(c)(a_1, \ldots a_{p+1})&=& (-1)^{l(a_1)l} a_1 c(a_2, \ldots , a_{p+1}) 
+ \sum_{i= 1}^{p} (-1)^i c(a_1, \ldots , a_ia_{i+1}, \ldots, a_{p+1}) \\&&+ (-1)^{p+1} c(a_1, \ldots , a_k) a_{p+1}.
\end{eqnarray*}
The vertical differential 
\begin{equation*}
d_v: C^{p,l}(\mathcal{A}, P)\rmap C^{p, l+1}(\mathcal{A}, P),
\end{equation*}
 is given by
\[ d_v(c)(a_1, \ldots , a_p)= d(c(a_1, \ldots , a_p))- \sum_{i= 1}^{p} (-1)^{\epsilon_i} c(a_1, \ldots , \delta(a_i), \ldots , a_p),\]
where $\epsilon_i= l+ l(a_1) + \ldots + l(a_{i-1})$. 
These two differentials commute and we will denote the resulting total complex by $C^{*}_{B}(\mathcal{A}, \mathcal{E})$.
For any $\zeta\in C^{p,l}_{B}(\mathcal{A}, \mathcal{E})$ we define $\bar{\zeta}\in C^{p,l}(\bar{\mathcal{A}}, \bar{\mathcal{E}})$ by
\[ \bar{\zeta}(a_1, \ldots, a_n)= (-1)^{\epsilon} \zeta(a_1, \ldots , a_n),\ \ \textrm{where}\ \epsilon= \sum_{1\leq i< j\leq n} k(a_i) |a_j|.\]
Note that this expression is well defined even if the $a_i$ are infinite sums, because the map $\zeta$ preserves the $k$ degree.
\begin{lemma} The map 
\[ C^{*}_{B}(\mathcal{A}, \mathcal{E})\rmap C^{*}(\bar{\mathcal{A}}, \bar{\mathcal{E}}),\ \zeta\mapsto \bar{\zeta}\]
is a morphism of cochain complexes.
\end{lemma}
\begin{proof}
That the horizontal differentials commute is straightforward. For the other direction, recall that the differential
$d_{tot}$ in $\overline{\mathcal{A}}$ is given by:
\begin{equation*}
d_{tot}=\partial + (-1)^n d,
\end{equation*}
where $d$ is the alternating sum of the operators $d_i$. The formula for the vertical differential in  $C^{*}(\bar{\mathcal{A}}, \bar{\mathcal{E}})$ decomposes in two pieces, one corresponding to $\partial$ and one corresponding to $(-1)^nd$. One can easily check that the first part
corresponds to the vertical differential in $C^{*}_{B}(\mathcal{A}, \mathcal{E})$, while the second part vanishes on $\bar{\zeta}$, because
$\zeta$ commutes with the operators $d_i$.
\end{proof}

The derivation $\delta: \Omega \rightarrow \mathcal{T}$ together with the component $T_0$ can now be interpreted
as a canonical Hochschild cochain of degree zero:
\[ \zeta^{\textrm{u}}:= \delta+ T_0 \in C^0(\Omega, \mathcal{T}).\] 
Using cup-product operations, one obtains new cochains in $C^0(\Omega_m, \mathcal{T}_m)$ as follows.
To simplify notations, we consider the case $m= 2$. Consider the two cocycles:
\[ \zeta^{1}= \zeta^{\textrm{u}}\boxtimes \textrm{Id} \in C(\Omega_2, \mathcal{T}\boxtimes \Omega), \]
\[ \zeta^{2}= \textrm{Id}\boxtimes \zeta^{\textrm{u}} \in C(\Omega_2, \Omega \boxtimes \mathcal{T}), \]
where $\textrm{Id}= \textrm{Id}_{\Omega}$ and where the operations
\[ (-)\boxtimes \textrm{Id}: \underline{Hom}_{B}^{*}(\Omega, \mathcal{T})\rmap \underline{Hom}_{B}^{*}(\Omega\boxtimes \Omega, \mathcal{T}\boxtimes \Omega),\]
\[ \textrm{Id} \boxtimes (-): \underline{Hom}_{B}^{*}(\Omega, \mathcal{T})\rmap \underline{Hom}_{B}^{*}(\Omega\boxtimes \Omega, \Omega\boxtimes \mathcal{T}),\]
are defined with the usual sign conventions. Using the composition
\[ \circ: (\mathcal{T}\boxtimes \Omega)\otimes  (\Omega\boxtimes \mathcal{T})\rmap \mathcal{T}\boxtimes \mathcal{T}= \mathcal{T}_2,\]
we now form the cup-product
\[ \zeta:= \zeta^{1}\cup \zeta^{2} \in C(\Omega_2, \mathcal{T}_2).\]
Using the construction from the last part of the appendix, we define
\[ x_{\omega}:= \bar{\zeta}(\omega) .\]
Lemma \ref{remark Hochschild} gives us the following.

\begin{proposition} 
For any $\omega\in \mathcal{MC}_2$, $x_{\omega}$ is an universal Maurer-Cartan morphism from $\omega$ to itself. 
\end{proposition}

\begin{remark}\label{non-commutative} Here is one final remark on the structures involved. In this paper we have thought of a representation
up to homotopy as a cochain complex of vector bundles $(E, \partial)$ together with 
the extra-data $\{R_k: k\geq 1\}$; the relevant algebraic structure was that of DB-algebra
and Maurer-Cartan elements with vanishing 0-component.  
One can follow a slightly different route, which has some advantages when it comes to the universal
Maurer-Cartan module: think of a representation up to homotopy as a graded vector bundle together with the 
extra-data $\{R_k: k\geq 0\}$. The relevant algebraic structure is that of B-algebra, which is defined exactly
as that of DB-algebra but giving up on the differential $\partial$ and requiring unitality. In terms of the formalism discussed above,
the resulting category $\underline{\mathcal{B}ar}$ of B-algebras coincides with the category
$\textrm{GrAlg}(\underline{\mathbb{VS}}_{B})$ of (unital) graded algebras associated to the monoidal category 
$\underline{\mathbb{VS}}_{B}$. Then, for a graded vector bundle $E$, $\mathcal{A}_{E}$ is a B-algebra
and representations up to homotopy on $E$ correspond to Maurer-Cartan elements of $\mathcal{A}_{E}$
(with no restriction on the zero-component).  As analogues of $\Omega$ and $\mathcal{T}$, one looks at
\begin{itemize}
\item $\Omega_{\alpha}\in \underline{\mathcal{B}ar}$ together with a Maurer Cartan element
$L_{\alpha}\in \textrm{MC}(\bar{\Omega}_{\alpha})$ which is universal among pairs $(\mathcal{A}, \omega)$ 
consisting of a Maurer Cartan element in a $B$-algebra.
\item $\mathcal{T}_{\alpha}$ which has the same universal property as $\mathcal{T}$, 
but for bimodules over $B$-algebras.
\end{itemize}
It is not surprising that one can explicitly construct $\Omega_{\alpha}$ out of $\Omega$ 
by adjoining to it a unit and a formal element $\alpha$ of bidegree $(0, 1)$:
\[\Omega_{\alpha}^{k}(l)=\Omega^{k}(l)+\alpha \circ \Omega^{k}(l-1),\]
except in bidegrees $(0, 0)$ and $(0, 1)$ where
\[\Omega_{\alpha}^{0}(0)= \mathbb{Q}, \Omega_{\alpha}^{0}(1)= \mathbb{Q}\alpha.\]
One defines the algebra structure on $\Omega_{\alpha}$ by requiring
\[ \alpha^2= 0,\ \ \alpha \circ a- (-1)^{l(a)} a\circ \alpha = \partial(a),\]
while the $d_i$'s are defined by 
\[ d_i(a+ \alpha\circ b)= d_i(a) + \alpha \circ d_i(b).\] 
Finally, one sets $L_{\alpha}= \alpha+ L$.

For $\mathcal{T}_{\alpha}$ the situation is similar but a bit simpler:
\[ \mathcal{T}_{\alpha}= \mathcal{T}+ \alpha\circ \mathcal{T}\]
and the differential $\partial$ of $\mathcal{T}$ is encoded in 
 the bimodule structure of $\mathcal{T}_{\alpha}$:
\[ x\circ \alpha= (-1)^{l(x)}( \alpha \circ x- \partial(x)),\]
for $x\in \mathcal{T}$. 

The analogue $\delta_{\alpha}: \Omega_{\alpha}\rmap \mathcal{T}_{\alpha}$ of the derivation $\delta$
has a nicer universal property: it is universal among all derivations on $\Omega_{\alpha}$-bimodules.
Using the straightforward B-version of Hochschild cohomology and non-commutative forms, 
we see that $\mathcal{T}_{\alpha}$ must be the space of non-commutative 1-forms
associated to the B-algebra $\Omega_{\alpha}$. This also gives another description 
of $\mathcal{T}_{\alpha}$ (and then of $\mathcal{T}$) out of $\Omega_{\alpha}$: 
\[ \mathcal{T}_{\alpha} = \Omega_{\alpha}\otimes \overline{\Omega}_{\alpha},  \ \ \textrm{where}\ \overline{\Omega}_{\alpha}=  \Omega_{\alpha}/1\cdot \mathbb{R},\]
and where a tensor $a\otimes b$ should be interpreted as a non-commutative 1-form $a\delta_{\alpha}(b)$.
Since $\mathcal{T}$ can be recovered as the subspaces of elements which are not of type $\alpha\circ x$, the derivation property of
$\delta_{\alpha}$ shows that $\mathcal{T}$ is spanned by the following types of elements:
\[ A\delta_{\alpha}(T) B,\]
where $A$ is either $1$ or an element of $\Omega$, similarly for $B$, and $T$ is either a tree or $\alpha$. This corresponds 
to our original description of $\mathcal{T}$ in terms of trees and forests. It is interesting to point out that the appearance of $\emptyset$
in that description encodes two types of elements: $1$ (on the forest side) and $\delta(\alpha)$ (on the tree side).
\end{remark}

\appendix

\section{Appendix}

\subsection{Maurer-Cartan elements}

In this appendix we put together some definitions and results that are used in the paper.

We begin with some standard notions regarding Maurer-Cartan elements in Differential Graded Algebras (DGAs).

\begin{enumerate}
\item A \textit{Maurer-Cartan element} in a DGA $(A,d)$ is an element $\gamma\in A$ of degree one satisfying $d(\omega)+\omega^{2}=0$. We denote by $MC(A)$ the set of all Maurer-Cartan elements. 
	
\item A \textit{gauge equivalence} between $\omega$ and $\theta\in MC(A)$ is an invertible element $u\in A$ of degree zero satisfying $u\omega u^{-1}-\theta=(du)u^{-1}$.  
	
\item Given two Maurer-Cartan elements $\theta$ and $\omega$ of two DGAs $(A,d)$ and $(B,d)$, respectively, and a DG $A$-$B$ bimodule $(P,d)$, a \textit{Maurer-Cartan $P$-morphism} from $\omega$ to $\theta$ is an element $x\in P$ of degree zero satisfying $x\omega-\theta x=d(x)$. We denote by $P(\omega,\theta)$ the set of such $P$-morphisms. 

\item With the same notations, we say that $x,y\in P(\omega,\theta)$ are \textit{homotopic} if there exists $h\in P$ of degree $-1$ such that $x-y=dh+h\omega+\theta h$. We denote by $P[\omega,\theta]$ the set of all homotopy classes of $P$-morphisms from $\omega$ to $\theta$. 
\end{enumerate}

Altogether, one obtains a category whose objects are DGAs endowed with a Maurer-Cartan element where the morphisms from $(B,\omega)$ to $(A,\theta)$ are pairs $(P,x)$ consisting of a DG $A$-$B$ bimodule $P$ and an element $x\in P(\omega,\theta)$. If $(Q,y)$ is another morphism from $(C,\eta)$ to $(B,\omega)$, then their composition is defined as 
\[
	(P,x)\circ(Q,y)=(P\otimes_{B}Q,x\otimes y).
\]
It is easy to check that $x\otimes y$ satisfies the required equation and also that this operation is compatible with the notion of homotopy. In particular, one obtains a quotient of this category in which homotopic morphisms become equal.\\

We will now concentrate our attention on complete DGAs and complete DG modules. By a filtered algebra we mean an algebra $A$ together with a filtration 
\[
	\cdots\subset F_{2}A\subset F_{1}A\subset F_{0}A=A,
\]
satisfying 
\[
	F_{p}A\cdot F_{q}A\subset F_{p+q}A.
\]
Note that, in particular, $F_{p}A$ is an ideal in $A$ for all $p$, hence we can consider the quotient algebras, which fit into a tower
\[
	A/F_{1}A\leftarrow A/F_{2}A\leftarrow\cdots.
\]
We denote by $\bar{A}$ the inverse limit of this tower. Note that $\bar{A}$ has a natural filtration, with $F_{p}\bar{A}$ being the inverse limit of
\[
	F_{p}A/F_{p+1}A\leftarrow F_{p}A/F_{p+2}A\leftarrow\cdots.
\]
Moreover, there is a canonical map $c:A\rmap\bar{A}$ that is a map of filtered algebras.

\begin{definition} 
A complete algebra is an algebra $A$ together with a filtration $F_{\bullet}A$, such that the canonical map $c:A\rmap\bar{A}$ is an isomorphism.

A complete DGA is a DGA $(A,d)$ that also has the structure of a complete algebra, such that each space $F_{p}A$ of the filtration is a subcomplex of $(A,d_{A})$. 
\end{definition}

A similar discussion applies to modules over filtered algebras. Given $A$ as above, a filtered left $A$-module $P$ is required to carry a filtration 
\[
	\cdots\subset F_{2}P\subset F_{1}P\subset F_{0}P=P,
\]
satisfying 
\[
	F_{p}A\cdot F_{q}P\subset F_{p+q}P.
\]
The completion $\bar{P}$ of $P$ is the inverse limit of $P/F_{p}P$ (a left $\bar{A}$-module). If $A$ is a complete algebra, we say that $P$ is a complete (left) $A$-module if the canonical map from $P$ to $\bar{P}$ is an isomorphism. If $(A,d)$ is a complete DGA, a complete (left) DG module over $(A,d)$ (or simply $A$-module) is a DG module $(P,d)$ that also has the structure of complete $A$-module such that each $F_{p}P$ is a subcomplex of $(P,d)$. Right modules and bimodules are treated similarly.

In general, for a filtered algebra $A$, $\bar{A}$ is complete and is called the completion of $A$.

\begin{example} \label{ex:filtration} 
If $A=\oplus_{k,l}A^{k}(l)$ is a (differential) bigraded algebra, then it can also be viewed as a (differential) graded algebra with $A=\oplus_{n}A^{n}$, where 
\[
	A^{n}=\oplus_{k+l=n}A^{k}(l).
\]
In this case, $A$ carries a natural filtration with 
\[
	F_{p}A=\oplus_{k\geq p}A^{k}(l).
\]
The resulting completion $\bar{A}$ is given by
\[
	\bar{A}^{n}=\Pi_{k+l=n}A^{k}(l).
\]
\end{example}

\subsection{The case of complete DGAs}

We now consider the existence problem for Maurer-Cartan elements whose class modulo $F_{r}A$ (for some $r\geq1$) is given. Let $\gamma\in A$ be of degree one, and suppose that we look for a Maurer-Cartan element $\omega$ that is equivalent to $\gamma$ modulo $F_{r}A$. This condition forces: 
\begin{equation}
	d\gamma+\gamma^{2}\equiv0\ \textrm{mod}F_{r}A.\label{forced}
\end{equation}
Any $\gamma\in MC(A)$ induces a new differential on $A$: 
\[
	d_{\gamma}(a)=d(a)+[\gamma,a]=d(a)+\gamma a-(-1)^{|a|}a\gamma.
\]
This differential descends to a differential $d_{\gamma}$ on all the quotients $F_{p}A/F_{p+1}A$. Actually, for $d_{\gamma}$ to be a differential on the quotients, one does not need the full Maurer-Cartan condition on $\gamma$ but only (\ref{forced}) for $r=1$. Hence given any $\gamma$ of degree $1$ satisfying (\ref{forced}) for some $r\geq1$, it makes sense to talk about the cohomology of $(F_{p}A/F_{p+1}A,d_{\gamma})$.

\begin{proposition}\label{MC1} 
Let $A$ be a complete DGA. Then for any degree one element $\gamma$ satisfying (\ref{forced}) and
\[
	H^{2}(F_{p}A/F_{p+1}A,d_{\gamma})=0,\ \ \ \forall\ \ p\geq r,
\]
there exists $\omega\in MC(A)$ such that $\omega\equiv\gamma\ \textrm{mod}\ F_{r}A$.
\end{proposition}

\begin{proof} 
We will inductively construct  $\omega_{r},\omega_{r+1},\ldots$ with the property that 
\[
	d\omega_{k}+\omega_{k}^{2}=0,\;\mod F_{k}A,
\]
in such a way that 
\[
	\omega_{k}=\omega_{k-1}\ \textrm{mod}\ F_{k-1}A,\ \forall\ k\geq r+1,\ \omega_{r}=\gamma.
\]
Assuming that $\omega_{k}$ has been constructed, we are now looking for $a\in F_{k}A$ such that 
\[
	\omega_{k+1}:=\omega_{k}+a
\]
satisfies the Maurer-Cartan equation modulo $F_{k+1}A$. Writing out the equation and using that $a^{2}\in F_{k+1}A$, the equation to solve is 
\[
	-d_{\omega_{k}}(a)=(d\omega_{k}+\omega_{k}^{2})\ \textrm{mod}\ F_{k+1}A.
\]
This can be seen as an equation in $F_{k}A/F_{k+1}A$. Moreover, on the quotient, $d_{\omega_{k}}=d_{\gamma}$ because $\omega_{k}-\gamma\in F_{1}(A)$. Hence, due to the cohomological condition in the statement, we only have have to check that the right hand side of the last equation is closed for $d_{\omega_{k}}$. But its differential (modulo $F_{k+1}A$) is 
\[
	d(\omega_{k}^{2}+d\omega_{k})+\omega_{k}(\omega_{k}^{2}+d\omega_{k})-(\omega_{k}^{2}+d\omega_{k})\omega_{k}=0.
\]
In conclusion, we obtain the desired sequence $(\omega_{k})_{k\geq r}$. Due to completeness of $A$, we obtain an element $\gamma\in A$ such that $\omega=\omega_{k}\ \textrm{mod}F_{k}A$ for all $k\geq r$. Since the Maurer-Cartan expression in $\omega$ is congruent, modulo
$F_{k}A$, to the one of $\omega_{k}$, hence to zero, we deduce (again from the completeness of $A$) that $\omega\in MC(A)$. By construction, $\omega=\gamma\ \textrm{mod}F_{r}A$. 
\end{proof}

There is an analogous result for Maurer-Cartan morphisms. Given two Maurer-Cartan elements $\omega$ and $\theta$ of two complete DGAs $(A,d)$ and $(B,d)$, respectively, and let $(P,d)$ be a a complete DG-$A$-$B$ bimodule. Then the differential $d$ of $P$ can be twisted by $\omega$ and $\theta$ to define a new differential:
\[
	d_{\omega,\theta}(x)=d(x)+\theta x-(-1)^{|x|}x\omega.
\]
The following is proven exactly as the previous result.

\begin{proposition}\label{MC11} 
Let $r\geq1$, and assume that 
\[
	H^{1}(F_{p}P/F_{p+1}P,d_{\omega,\theta})=0,\ \ \ \forall\ \ p\geq r.
\]
Then for any $x\in P$ satisfying 
\[
	x\omega-\theta x=dx\ \textrm{mod}\ F_{r}P,
\]
one can find $y\in P(\omega,\theta)$ such that $y=x\ \textrm{mod}\ F_{r}P$. Moreover, if the same cohomological condition holds in degree zero, then any two such $y$'s are homotopic. 
\end{proposition}

\textbf{Gauge equivalence in complete DGAs:} 
In the context of complete DGAs, there is a refined notion of gauge transformation that we now explain. We associate a group $G_{1}(A)$ to a complete DGA $A$ as follows
\[
	G_{1}(A)=\{x\in A^{0}:\ x\equiv1\ \textrm{mod}\ F_{1}A\}=1+(F_{1}A)^{0}.
\]
One can see that $G_{1}(A)$ is a group with respect to the multiplication in $A$ from the power series expression 
\[
	(1-\alpha)^{-1}=\sum_{k\geq0}\alpha^{k},
\]
where, for $\alpha\in F_{1}A$, completeness implies that the right hand side makes sense as an element of $A$ . 

Note that, strictly speaking, the definition of $G_{1}(A)$ requires $A$ to be unital. However, the role of the elements {}``$1$'' is purely formal. In other words, $G_{1}(A)$ makes sense even without the unitality condition. Equivalently, taking this as a definition for unital $A$'s, for a general $A$, one can replace $A$ by the new (complete) DGA $A^{+}$ obtained by adding a unit to $A$. Define $G_{1}(A)$ as $G_{1}(A^{+})$. In the case that $A$ already has a unit $1_{A}$, the map $1_{A}+x\mapsto1+x$ identifies the two definitions.

\begin{definition}\label{strong-appendix} 
Given a complete DGA $A$, a gauge equivalence $u$ between two Maurer-Cartan elements $\omega$ and $\theta$ of $A$ is called strong if $u\in G_{1}(A)$. If such a $u$ exists, we say that $\omega$ and $\theta$ are strongly gauge equivalent. 
\end{definition}

\begin{proposition}\label{MC2} 
Let $A$ be a complete DGA and $\omega,\theta\in MC(A)$, such that 
\[
	\omega=\theta\ \textrm{mod}\ F_{r}A,
\]
with $r\geq1$. If 
\[
	H^{1}(F_{p}A/F_{p+1}A,d_{\omega})=0,\ \ \ \forall\ \ p\geq r,
\]
then $\omega$ and $\theta$ are strongly gauge equivalent. 
\end{proposition}

\begin{proof} 
This  proof is very similar to the one of Proposition \ref{MC1}. We construct inductively a sequence $u_{r},u_{r+1},\ldots$ of degree zero elements of $A$ with the property that 
\[
 u_{k}\omega-\theta u_{k}=du_{k}\ \textrm{mod}\ F_{k}A
\]
for all $k\geq r$. Moreover, the sequence will be constructed so that 
\[
 u_{k}=u_{k-1}\ \textrm{mod}\ F_{k-1}A\ \forall\ k\geq r+1,\ \ u_{r}=1.
\]
Assuming that $u_{k}$ has been constructed, we are looking for $x\in F_{k}A$ such that 
\[
 (u_{k}+x)\omega-\theta(u_{k}+x)=d(u_{k}+x)\ \textrm{mod}\ F_{k+1}A.
\]
Note that, since $\omega-\theta\in F_{r}A\subset F_{1}A$, 
\[
 \theta x-x\omega=\omega x-x\omega\ \textrm{mod}\ F_{k+1}
\]
whenever $x\in F_{k}A$. We see that the previous equation can be written as an equation in $F_{k}A/F_{k+1}A$: 
\[
 d_{\omega}x=-du_{k}+u_{k}\omega-\theta u_{k}\ \textrm{mod}\ F_{k+1}A.
\]
Because of the hypothesis, it suffices to show that the right hand side is closed in the quotient. Denoting the right hand side by $y$, and using that 
\[
 d_{\omega}(y)=d(y)+\theta y+y\omega,
\]
the desired equation follows immediately.

With the sequence $u_{r},u_{r+1},\ldots$ constructed, one uses again the completeness of $A$ to obtain $u\in A$ of degree zero such that $u=1\ \textrm{mod}F_{r}A$, $u\omega-\theta u=du$.
\end{proof}

\begin{remark} 
The gauge equivalence comes from an action of $G_{1}(A)$ on $MC(A)$, given by the usual gauge formula:
\[
  u\cdot\omega=u\omega u^{-1}-du\cdot u^{-1}\ \ u\in G_{1}(A),\omega\in MC(A).
\]
Our discussion has an infinitesimal counterpart. First of all, ``the Lie algebra of $G_{1}(A)$'' is defined as
\[
 \mathfrak{g}_{1}(A):=\{\alpha\in A^{0}:\alpha\equiv0\ \textrm{mod}\ F_{1}A\},
\]
with the commutator bracket 
\[
 [\alpha,\beta]=\alpha\beta-\beta\alpha.
\]
The exponential map 
\[
 exp:\mathfrak{g}_{1}(A)\rmap G_{1}(A)
\]
is defined by the usual power series 
\[
 exp(\alpha)=\sum_{k\geq0}\frac{1}{k!}\alpha^{k}.
\]
The completeness of $A$ make sense of $exp(\alpha)$ for $\alpha\in F_{1}A$. The action of $G_{1}(A)$ on $MC(A)$ has an infinitesimal counterpart: an action of the Lie algebra $\mathfrak{g}_{1}(A)$ on $MC(A)$, which is familiar in the discussion of Maurer-Cartan elements in differential graded Lie algebras: 
\[
 \alpha\cdot\gamma=[\alpha,\gamma]+d\alpha,\ \ \alpha\in\mathfrak{g}_{1}(A),\gamma\in MC(A).
\]
\end{remark}

\subsection{Relation with Hochschild cohomology}\label{Hochschild}

Let $A$ be a DGA and let $P$ be a $A$-bimodule. Here we explain that, given a Maurer-Cartan element $\omega$ of $A$, one can associate a Maurer-Cartan morphism to a degree zero $P$-valued Hochschild cocycle on $A$. In low degrees, the idea is very simple: by applying a biderivation $D:A\rmap P$ to the Maurer-Cartan equation for a Maurer-Cartan element $\omega$, one obtains an element $D(\omega)\in P(\omega,\omega)$. We now discuss what happens for general degree zero Hochschild cocycles. For each $k$ and $l$, we denote by $C^{k,l}(A,P)$ the space of all linear maps 
\[
 c:\underbrace{A\otimes\ldots\otimes A}_{k\ \textrm{times}}\rmap P
\]
which raises the total degree by $l$. The horizontal differential
\[
 b:C^{k,l}(A,P)\rmap C^{k+1,l}(A,P)
\]
is given by 
\begin{eqnarray*}
 b(c)(a_{1},\ldots a_{k+1}) & = & (-1)^{|a_{1}|l}a_{1}c(a_{2},\ldots,a_{k+1})
          +\sum_{i=1}^{k}(-1)^{i}c(a_{1},\ldots,a_{i}a_{i+1},\ldots,a_{k+1})\\
      &  & +(-1)^{k+1}c(a_{1},\ldots,a_{k})a_{k+1}.
\end{eqnarray*}
The vertical differential 
\[
 d_{v}:C^{k,l}(A,P)\rmap C^{k,l+1}(A,P)
\]
is given by 
\[
 d_{v}(c)(a_{1},\ldots,a_{k})=d(c(a_{1},\ldots,a_{k}))-\sum_{i=1}^{k}(-1)^{\epsilon_{i}}c(a_{1},\ldots,d(a_{i}),\ldots,a_{k}),
\]
where $\epsilon_{i}=l+|a_{1}|+\ldots+|a_{i-1}|$. These two differentials commute, and we obtain a complex $C^{n}(A,P)=\bigoplus_{k+l=n}C^{k,l}(A,P)$ with  $D=d_{v}+(-1)^{l}b$ as total differential. We are interested in $0$-cocycles. Such a cocycle is a finite sum 
\begin{equation}
 \zeta=\zeta_{0}+\zeta_{1}+\ldots,\ \ \ \textrm{with}\ \ \zeta_{k}\in C^{k,-k}(A,P),\label{zeta-H}
\end{equation}
satisfying 
\[
 b(\zeta_{i})+(-1)^{i}d(\zeta_{i+1})=0.
\]
For any such $\zeta$, we consider the induced polynomial function
\[
 \hat{\zeta}:A^{1}\rmap P^{0},\ \ \zeta(a)=\zeta_{0}+\zeta_{1}(a)+\zeta_{2}(a,a)+\ldots\ \ .
\]
The following is straightforward:
 
\begin{lemma}\label{remark Hochschild}
For any Hochschild cocycle $\zeta\in C^{0}(A,P)$ and any $\omega\in MC(A)$, $\hat{\zeta}(\omega)\in P(\omega,\omega)$. Moreover:
\begin{itemize}
\item If $\zeta$ and $\zeta'$ are cohomologous, then $\zeta(\omega)$ and $\zeta'(\omega)$ are homotopic.
\item This construction is compatible with cup-products. 
\end{itemize}
\end{lemma}


\begin{thebibliography}{10}

\bibitem{AC3} C. Arias Abad and M. Crainic, {\em Representations up to homotopy and Bott's spectral sequence for Lie groupoids}, preprint arXiv:0911.2859.

\bibitem{Block-Smith} J. Block and A. Smith, {\em A Riemann Hilbert correspondence for infinity local systems}, preprint arXiv:0908.2843.

\bibitem{Bott} R. Bott, {\em On the Chern-Weil homomorphism and continuous cohomology of Lie groups}, Adv. Math. {\bf 11} (1973), 289--303.

\bibitem{Beh1} K. A. Behrend, {\em On the de Rham cohomology of differential and algebraic stacks}, Adv. Math. {\bf198}(2) (2005), 583--622, .

\bibitem{Get} E. Getzler, {\em The equivariant Chern character for non-compact Lie groups}, Adv. Math. {\bf 109}(1) (1994), 88--107.

\bibitem{Forcey} S. Forcey, {\em Quotients of the multiplahedron as categorified associahedra}, Homology Homotopy Appl. {\bf10}(2) (2008), 227--256.

\bibitem{Igusa} K. Igusa, {\em Iterated integrals of superconnections}, preprint arXiv:0912.0249.

\bibitem{KMM} T. Kaczynski, K. Mischaikow and M. Mrozek,{\em Computational Homology}, Applied Mathematical Sciences Vol. {\bf 157} (2004), Springer.

\bibitem{Loday} J.L. Loday, {\em The diagonal of the Stasheff polytope}, to appear in the proceedings of the International Conference in honor of Murray Gerstenhaber and Jim Stasheff (Paris 2007).

\bibitem{Mackenzie} K.C. H. Mackenzie,{\em General theory of Lie groupoids and Lie algebroids},  London Math. Soc. Lecture Note Series {\bf 213} (2005), Cambridge University Press.

\bibitem{MacLane} S. Mac Lane, {\em Categories for the Working Mathematician},  second edition,  Graduate Texts in Mathematics {\bf 5} (1998), Springer. 

\bibitem{M} M. Markl, {\em Homotopy algebras are homotopy algebras}, Forum Math. {\bf 16}(1) (2004), 129--160.

\bibitem{MS} M. Markl and S. Shnider, {\em Associahedra, cellular W-construction and products of $A_{\infty}$-algebras}, Trans. Amer. Math. Soc. {\bf 358}(6) (2006), 2353--2372, .

\bibitem{MSS} M. Markl, S. Shnider and J. Stasheff, {\em Operads in algebra, topology and physics}, Mathematical Surveys and Monographs {\bf 96} (2002), American Mathematical Society.

\bibitem{Serre} J.P. Serre, {\em Homologie singuliere des espaces fibres}, (French), Ann. of Math. {\bf 54}(2) (1951), 425--505 .

\bibitem{Stasheff} J.D. Stasheff, {\em Homotopy associativity of \$H\$-spaces. I, II}, Trans. Amer. Math. Soc. {\bf 108} (1963), 275--292, 293--312.

\bibitem{Stasheff2} J.D. Stasheff, {\em A twisted tale of cochains and connections},  Georgian Mathematical Journal {\bf 17}(1), 203--215.

\bibitem{SU} S. Saneblidze and R. Umble, {\em Diagonals on the permutahedra, multiplihedra and associahedra}, Homology Homotopy Appl. {\bf 6}(1) (2004), 363--411.

\bibitem{Sug} M. Sugawara, {\em On the homotopy commutativity of groups and loop spaces}, Mem. Coll. Sci. Univ. Kyotp Ser. A Math. {\bf 33} (1960), 257--269 .
\end{thebibliography}
\end{document}